\documentclass[12pt]{article}
\usepackage[body={6.5in,9.0in},left=1in,top=1in,centering]{geometry}
\usepackage[letterpaper=true,colorlinks=true,urlcolor=blue,citecolor=blue,linkcolor=blue,pdfstartview=FitH]{hyperref}

\usepackage{amsmath,mathrsfs,amsfonts,amssymb,graphicx,pstricks}
\usepackage{amssymb,mathrsfs,amsthm,epstopdf,epsfig}
\usepackage{authblk,pstricks}
\newcommand{\specialcell}[1]{\ifmeasuring@#1\else\omit$\displaystyle#1$\ignorespaces\fi}
 \usepackage[sort,comma]{natbib}
\allowdisplaybreaks
\newtheorem{theorem}{Theorem}[section]
\newtheorem{corollary}[theorem]{Corollary}
\newtheorem{lemma}[theorem]{Lemma}
\newtheorem{prop}[theorem]{Proposition}
\theoremstyle{definition}
\newtheorem{assumption}[theorem]{Assumption}

\newtheorem{rem}[theorem]{Remark}
\newtheorem{example}[theorem]{Example}

\newcommand{\comment}[1]{}
\newcommand{\R}{\mathbb R}

\def\P{\mathbb P}
\def\d{\mathrm d}
\newcommand{\A}{\mathscr A}

\newcommand{\E}{\mathbb E}

\newcommand{\LL}{\mathcal L}
\newcommand{\N}{\mathbb N}

\newcommand{\e}{\varepsilon}
\newcommand{\sg}{\sigma}
\newcommand{\vphi}{\varphi}
\newcommand{\cadlag}{\text{c\`adl\`ag} }

\newcommand{\cd}{(\cdot)}
\newcommand{\1}{\mathbf 1}
\newcommand{\hx}{\hat x}

\makeatletter \@addtoreset{equation}{section}

\usepackage{pgfplots}
\begin{document}
	
	\title{On an Ergodic  Two-Sided  Singular Control Problem}
	
\comment{	 \author[1]{Khwanchai Kunwai}
	 \author[2]{Fubao Xi}
	 \author[3]{George Yin}
	 \author[1]{Chao Zhu}
	
	\affil[1]{Department of Mathematical Sciences,   University of Wisconsin-Milwaukee,   Milwaukee, WI 53201,   USA,   {\tt kkunwai@uwm.edu}, {\tt zhu@uwm.edu}}
	\affil[2]{School of Mathematics and Statistics, Beijing Institute of Technology, Beijing 100081, China, {\tt xifb@bit.edu.cn}.}
	\affil[3]{Department of Mathematics,  University of Connecticut,   Storrs, CT 06269,  USA,  {\tt 	gyin@uconn.edu}}
	 \renewcommand\Authands{ and }}
	
\author{Khwanchai Kunwai,\thanks{Department of Mathematics, Faculty of Science, Chiang Mai University, Chiang Mai, 50200 Thailand,  {\tt khwanchai.kun@cmu.ac.th}}   \quad Fubao Xi,\thanks{School of Mathematics and Statistics, Beijing Institute of Technology, Beijing 100081, China, {\tt xifb@bit.edu.cn}. } 
 \quad George Yin,\thanks{Department of Mathematics,  University of Connecticut,   Storrs, CT 06269,  USA,  {\tt gyin@uconn.edu}.} 
  \quad and \quad Chao Zhu\thanks{Department of Mathematical Sciences,   University of Wisconsin-Milwaukee,   Milwaukee, WI 53201,   USA,   {\tt zhu@uwm.edu}. } 
  }

	\maketitle
	
\begin{abstract}
Motivated by applications in natural resource management, risk management,  and finance,  this paper is  focused on an ergodic two-sided singular control problem for a general one-dimensional diffusion process. The control is given by a bounded variation process.  Under some mild conditions, the optimal reward value as well as an optimal control policy are derived by the vanishing discount method. Moreover, the Abelian and Ces\`aro limits are established. Then a direct solution approach is provided at the end of the paper.

\medskip\noindent{\bf Key Words.} Ergodic singular control, vanishing discount, Abelian limit, Ces\`aro limit.
		
\medskip\noindent {\bf  2000 MR Subject Classification.} 93E20, 60H30, 60J70.
\end{abstract}
	
\section{Introduction}\label{sect:intro}
	
This work is motivated by
applications in   reversible investment problem (\cite{GuoP-05,DeAngelisF-14}), optimal harvesting and renewing problems (\cite{HeniT-20,HeniTPY-19}), and dividend payment and capital injection problems (\cite{JinYY-13,LindL-20}). In the aforementioned applications, the systems of interests are controlled by   bounded variation processes in order to achieve certain economic benefits. It is well known that any bounded variation process can be written  as a difference of two nondecreasing and \cadlag processes. Often the two nondecreasing processes will  introduce rewards and costs to the optimization  problems, respectively. For example, in  natural renewable resource management problems such as forestry, while harvesting brings profit, it is costly to renew the natural renewable resource.  One needs  to balance the harvesting and renewing decisions so as to achieve an optimal  reward.  Similar considerations prevail in  reversible investment and dividend payment and capital injection problems. In terms of the reward (or cost) functional, the two singular control processes have different signs. We call such problems    two-sided mixed singular control problems. In contrast, in  the   traditional singular control formulation, the reward or cost corresponding to the singular control is expressed in terms of the expectation of its total variation.

We note that most of the existing works  on two-sided mixed singular control problems are focused on discounted criterion.
  In many applications, however, optimizations using  discounted criteria  are not appropriate. For example, in optimal harvesting problems, discounted criteria are largely in favor of current interests and
   disregard the future effect. As a result, they
   may lead to myopic harvesting policy and extinction of the species. We refer to \cite{L-Oksendal,Song-S-Z,A-Shepp} for   examples in which the optimal policy under the discounted criterion is to harvest all at time 0, resulting in an immediate extinction.

 In view of these considerations, this paper aims to investigate a two-sided mixed singular control problem for a general one-dimensional diffusion on $[0,\infty)$ using a long-term average criterion. Ergodic singular control problems have been extensively studied in the literature;  see, for example, \cite{Hynd-13,JackZ-06,Kara-83,MenaR-13,Weera-02,Weerasinghe-07,WuChen-17}  and many others. In the aforementioned references, the cost (or reward) associated with the singular control  is expressed in terms of
expectation of the total variation.  Our formulation in \eqref{def: lambda_0} is different. In particular,  in  \eqref{def: lambda_0}, while $\eta(T)$ contributes positively toward the reward,  the  contribution from $\xi (T)$  is negative.
On the  one hand, this is motivated  by the specific applications in areas such as reversible investment,  harvesting and renewing, and dividend payment and capital injection problems.  On the other hand, the mixed signs also add an interesting twist to   singular control  theory as the gradient constraints in the corresponding HJB equation are different from those in the traditional setup as those in \cite{Kara-83,MenaR-13,Weera-02,Weerasinghe-07}.

To find the optimal value $\lambda_{0}$ of \eqref{def: lambda_0} and an optimal control policy, one may attempt to use the guess-and-check approach. That is, one  first finds a smooth solution to the HJB equation \eqref{e:HJB-lta}  and then uses   a verification argument  to derive the optimal long-term average value as well as an optimal control.  Indeed this is the approach used in \cite{Kara-83}, in which  the Abelian and Ces\`aro limits are established  for a singular control problem when the underlying process is a one-dimensional Brownian motion. The result was further extended to an ergodic   singular control problem for a one-dimensional diffusion process  in \cite{Weera-02} in which the drift and diffusion coefficients satisfy certain symmetry properties.

While this approach seems   plausible,  it is not easy to guess the right solution. In particular, since the underlying process $X_{0}$ is a general one-dimensional diffusion on $[0,\infty)$ (to be defined in \eqref{e:uncontrolled-SDE}) and the rewards associated with the singular controls $\xi$ and $\eta$ have mixed signs in \eqref{def: lambda_0},   our problem does not have the same  kind of symmetry as in \cite{Kara-83} and \cite{Weera-02}. In addition, the gradient constraint $c_{1} \le u'(x) \le c_{2}$ brings much difficulty and subtlety in finding the free boundaries that separates the action and non-action regions. Besides, this approach does not reveal how the ergodic, discounted and finite-time horizon problems (to be defined in \eqref{def: lambda_0}, \eqref{e-V_r(x) defn}, and \eqref{e-V_T(x) defn}, respectively) are related to each other.
	
In view of the above considerations,  we adapt  the vanishing discount  approach developed in \cite{MenaR-13} and \cite{Weerasinghe-07}.   First, we observe that  under Assumption \ref{assumption-pi_12},  
careful  analysis using reflected diffusion process on an appropriate interval $[a, b] $ reveals that $\lambda_{0} > 0$. We next use Assumptions \ref{assumption-V_r}  and \ref{assumption-mu-sigma}
    to show that $\lim_{r\downarrow 0} r V_{r}(x) = \lambda_{0}$ for any $x \ge 0$ and that there exist two positive constants $a_{*} < b_{*}$ for which the reflected diffusion process  on $[a_{*}, b_{*}]$ is an optimal state process, where $V_{r}(x)$ is  the value function of the discounted problem \eqref{e-V_r(x) defn}. These results are summarized in Theorem \ref{thm-lta-main result}.  Furthermore, we show in Theorem \ref{thm-cesaro-limit} that the Ces\`aro limit holds as well. As
      an illustration, we study an ergodic two-sided singular control problem for a geometric Brownian motion model using this framework in Section  \ref{sect-gbm}.  This approach fails when some of our assumptions are violated. Indeed, the case study  in subsection \ref{sec-mu=0} indicates that the long-term average reward can be arbitrarily large under suitable settings. 

Upon the completion of this paper, we learned the recent paper \cite{Alva-18}, which deals with an ergodic two-sided singular control problem for a general one-dimensional diffusion on $(-\infty, \infty)$.  The formulation in \cite{Alva-18} is different from ours because the rewards associated with the singular controls are both   positive. Under certain conditions, the paper first constructs a solution to the associated HJB equation  and then verifies that the local time reflection policy is optimal.   The approach is different from the vanishing discount method used in this paper, which also helps to establish the Abelian and Ces\`aro limits.

 Thanks  to  the referee who also brought our attention to the paper \cite{AlvaH-19},
which  
studies the optimal sustainable harvesting of a population that lives in a random environment. It proves that there exists a unique optimal  local time reflection harvesting strategy and establishes an Abelian limit under certain conditions. In contrast to the problem considered in this paper,
 \cite{AlvaH-19} is focused on
 a one-sided singular control problem without running rewards and  hence  it is simpler than our formulation in \eqref{def: lambda_0}.

Motivated by the two papers mentioned above, we add Section \ref{sect-direct-soln} to
 provide a direct solution approach to \eqref{def: lambda_0}.
  Following the idea in \cite{Alva-18} and \cite{AlvaH-19},  we first  impose conditions so that the long-term average reward for an $(a,b)$-reflection policy $\lambda(a, b)$ achieves its   maximum value  $\lambda_{*} = \lambda(a_{*}, b_{*})$ at a pair $0< a_{*} < b_{*}< \infty$. The maximizing pair $(a_{*}, b_{*})$  further allows us to derive a $C^{2}$ solution to the HJB equation \eqref{e:HJB-lta}.  This, together with the  verification  theorem (Theorem \ref{thm-verification-lta}), reveals that $\lambda_{0} = \lambda_{*}$ and the $(a_{*}, b_{*})$-reflection policy is optimal.

The rest of the paper is organized as follows.  Section \ref{sect-formulation} begins with the formulation of the problem;  it also presents    the main results of the paper.  Section \ref{sect-preliminaries}   collects some preliminary results. The Abelian and Ces\`aro limits are established in Sections \ref{sect-abelian limit} and \ref{sect-cesaro limit}, respectively. Section \ref{sect-gbm} is devoted to an ergodic  two-sided  singular control problem when the underlying process is a geometric Brownian motion with subsection \ref{sec-mu=0} providing a case study for $\mu =0$ and $h(x) = x^{p}, p\in (0,1)$.  Finally, we provide a direct solution approach in Section \ref{sect-direct-soln}. An example on ergodic  two-sided  singular control for Verhulst-Pearl diffusion  is studied in Section \ref{sect-direct-soln} for illustration.   	
 	
	\section{Formulation and Main Results}\label{sect-formulation}	
		To begin, suppose the uncontrolled process is given by the following one-dimensional diffusion process with state space $[0, \infty)$:
\begin{equation}\label{e:uncontrolled-SDE}
\d X_{0}(s)= \mu (X_{0}(s))\d s +\sigma (X_{0}(s))\d W(s), \quad X_{0}(0) = x_{0} \in [0, \infty),
\end{equation} where $W$ is a one-dimensional standard Brownian motion, and  $\mu$ and $\sigma $ are suitable functions so that a weak solution $X_{0}$ exists and is unique in the sense of probability law. We refer to Section 5.5 of \cite{Karatzas-S} for such conditions. Assume throughout the paper that  0 is an unattainable boundary point (entrance or natural) and that  $\infty$ is a  natural boundary point; see Chapter 15 of \cite{KarlinT81} or Section 5.5 of \cite{Karatzas-S} for more details on classifications of boundary points for one-dimensional diffusions. Note that if 0 is an entrance point, then it is part of the state space for $X_{0}$; otherwise, if 0 is a natural point, then it is not in the state space. While an entrance point cannot be reached from the interior, it is possible that the process $X_{0}$ will start from an entrance boundary point and quickly move to the interior and never return to it.  The choice of $[0, \infty)$ as the state space for $X_{0}$ is motivated by the following consideration.
The states of interest in  applications such as reversible investment, optimal harvesting and renewing, and dividend payment and capital injection problems are all bounded from below. For notational convenience, we then choose $[0, \infty)$ as the state space for $X_{0}$.

Throughout the paper,  we further suppose   that the process $X_{0}$ possesses   scale and speed densities given by
\begin{displaymath}
s(x) : = \exp\bigg\{-\int_{1}^{x} \frac{2\mu(y)}{\sigma^{2} (y) } \d y\bigg\}, \  \ \ m(x): = \frac{1}{\sigma^{2}(x) s(x)}, \qquad x > 0,
\end{displaymath} respectively. The scale and speed measures of $X_{0}$ are   $$  S[a, b]: = \int_{a}^{b} s(x) \d x, \quad\text{ and } \quad M[a, b]: = \int_{a}^{b} m(x) \d x, \quad \forall [a, b] \in [0, \infty).$$
  The infinitesimal generator of the process $X_{0}$ is 
	$$\LL f(x): = \frac12\sigma^{2}(x) f''(x) + \mu (x) f'(x)= \frac12 \frac{\d}{\d M}\bigg(\frac{\d f(x)}{\d S} \bigg), \ \quad \forall f\in C^{2}([0,\infty)).$$

We now  introduce a bounded variation process $\vphi = \xi-\eta$ to \eqref{e:uncontrolled-SDE}, resulting in the following  controlled dynamics:
\begin{equation}\label{e:controlled-sde}
		\begin{cases}\d X(s)= \mu (X(s))\d s +\sigma (X(s))\d W(s) +\d\xi(s) - \d\eta(s), \quad s\ge 0,\\
	X(0-)=x \ge 0. \end{cases}
	\end{equation} Throughout the paper, we assume that
	the  control process $\vphi(\cdot)= \xi(\cdot)- \eta(\cdot)$ is  an adapted, \cadlag process that admits the minimal Jordan decomposition $\vphi(t) = \xi (t) - \eta(t)$, $t\ge 0$.  In particular, $\xi, \eta$ are nonnegative and nondecreasing processes satisfying $\xi(0-) =\eta(0- )= 0$  such that the associated Borel measures  $\d\xi$ and $\d\eta$ on $[0,\infty)$ are mutually singular.     In addition, it is required that  under the control process $\vphi(\cdot)$,  \eqref{e:controlled-sde}  admits a unique nonnegative weak solution $X(\cdot)$.   Such a control process $\vphi(\cdot)$ is said to be {\em admissible}.
	
The   goal is to maximize the expected long-term average reward: 
	\begin{equation}
	\label{def: lambda_0}
	\lambda_{0} : =\sup_{\vphi(\cdot)\in \A_{x}} \liminf_{T\to\infty}\frac1T\E_{x}\bigg[\int_{0}^{T} h(X(s) ) \d s +c_{1} \eta(T)- c_{2}    \xi(T) \bigg],
	\end{equation} where $c_{1} < c_{2}$ are two positive constants, $h$ is a nonnegative function, 
	and  $\A_{x}$ is the set of {\em admissible controls}, i.e.,
	\begin{equation}\label{e:set A defn}
\begin{aligned}  
 \A_{x}: = \big\{\vphi  =(\xi,\eta) &\text{ is admissible and satisfies }   { \E_{x}[\xi(T) ] \le K_{1}(x) T^{n} + K_{2}(x) } \text{ for all } T >0\},
             \end{aligned}\end{equation} where $K_{ 1}(x) $ and $K_{2}(x)$ are positive real-valued functions,   and $n \in \N$ is a positive integer.
The requirement that $  \E_{x}[\xi(T) ] \le K_{1}(x) T^{n} + K_{2}(x) $ 
ensures that  
the expectation in the right-hand side of  \eqref{def: lambda_0} as well as the discounted and finite-time horizon problems \eqref{e-V_r(x) defn} and \eqref{e-V_T(x) defn} are well-defined.
   The set is clearly nonempty because the ``zero control'' $\vphi(t) \equiv 0 $ is in $\A_{x}$. Lemma \ref{lem:(LaLb)admissible} below  indicates that $\A_{x}$
   includes local time controls as well.  It is also apparent that the value of $\lambda_{0}$ does not depend on the initial condition $x$ since an initial jump dose not alter the value of the  limit in \eqref{def: lambda_0}. In addition, it is obvious that $\lambda_{0} \ge 0$.
	
	In this paper, we aim   to find the value $\lambda_{0}$ and an optimal control policy $\vphi^{*}$ that achieves the value $\lambda_{0}$. Motivated by \cite{Weerasinghe-07}, we will approach this problem via the vanishing discount method and  show that $\lambda_{0}$ is equal to the {\em Abelian limit} as well as the {\em Ces\`aro limit}. In other words, we
 demonstrate that   \begin{equation}
\label{e:abel-cesaro limits}
  \lambda_{0} = \lim_{r\downarrow 0} r V_{r}(x)= \lim_{T\to\infty} \frac{V_{T}(x)}{T},
\end{equation}  where $V_{r}(x)$ and $V_{T}(x)$ denote respectively  the value functions for the related discounted  and finite horizon problems \begin{align}
\label{e-V_r(x) defn}V_{r}(x) :& = \sup_{\vphi (\cdot)\in \A_{x}} \E_{x}\bigg[\int_{0}^{\infty} e^{-rs}[ h(X(s)) \d s + c_{1} \d \eta(s) - c_{2}\d \xi(s)]\bigg],\\
\label{e-V_T(x) defn}V_{T}(x) :& = \sup_{\vphi(\cdot)\in \A_{x}}\E_{x}\bigg[\int_{0}^{T} h(X(s) ) \d s +c_{1} \eta(T)- c_{2}    \xi(T) \bigg].
\end{align}  	
	The main result of this paper is  given next.


\begin{theorem}\label{thm-main result}
Under Assumptions \ref{assumption-pi_12}, \ref{assumption-V_r},  and  \ref{assumption-mu-sigma},    
the following assertions hold:
\begin{itemize}
 \item[(i)] There exist $0 < a_{*} < b_{*} < \infty$ so that the reflected diffusion    process on the  interval  $[a_{*},  b_{*}]$ $($if the initial point is outside this interval, then there will be an initial jump to the nearest point of the interval$)$ is an optimal state process for the ergodic control problem \eqref{def: lambda_0}. Hence the optimal control policy
is given by $\vphi_{*} = L_{a_{*}} - L_{b_{*}}$, in which $L_{a_{*}} $ and $L_{b_{*}}$ denote the local time processes at $a_{*}$ and $b_{*}$, respectively.
  \item[(ii)] The Abelian and Ces\`aro limits in \eqref{e:abel-cesaro limits} hold.
   \end{itemize}
\end{theorem}

 The proof of
Theorem \ref{thm-main result} follows from Theorems \ref{thm-lta-main result} and \ref{thm-cesaro-limit},  which will be presented in the subsequent sections.

 \begin{rem}{\rm
Since the optimal state process is a reflected diffusion on the interval $[a_{*},  b_{*}]\subset (0, \infty)$, it follows that it possesses a unique   invariant measure $\pi$. According to Chapter 15 of \cite{KarlinT81}, the   invariant measure is   $\pi(\d x) = \frac{1}{M[a_{*}, b_{*}]} M( \d x)$.
  Moreover, using the ergodicity
for linear diffusions in Chapter II, section 6 of \cite{boro:02}, we have $$\lambda_{0} = \int_{a_{*}}^{b_{*}} h(x) \pi(\d x)+ \frac{c_{1}}{2M[a_{*}, b_{*}] s(b_{*})} -\frac{c_{2}}{2M[a_{*}, b_{*}] s(a_{*})} .$$
}\end{rem}

\section{Preliminary Results}\label{sect-preliminaries}
In this section, we  provide some preliminary results.

	\begin{theorem}\label{thm-verification-lta}
		Suppose there exists a nonnegative function $u\in C^{2}([0,\infty))$ and a nonnegative number $\lambda$ such that
		\begin{equation}
		\label{e:HJB-lta}
		\max\big\{   \LL u( x) + h(x)-\lambda, \ u' (x) - c_{2}, \  -u' (x) + c_{1} \big \} = 0.
		\end{equation}
		Then $\lambda_{0} \le \lambda$.
	\end{theorem}

	\begin{proof}	Let $x\in [0,\infty)$ and $\vphi(\cdot)=(\xi\cd,\eta\cd)\in \A_{x}$ be an arbitrary control policy and denote by $X$ the controlled state process with $X(0-)=x$. For $n\in \N$, let $\beta_{n}: = \inf\{t\ge 0: X(t) \ge n\}$.  By Ito's formula we have
		\begin{align}\label{e0:verification-proof-lta}
		\nonumber	u(X(T\wedge \beta_{n})) &= u(x) + \int_{0}^{T\wedge \beta_{n}} \LL u(X(s))\d s - \int_{0}^{T\wedge \beta_{n}}u'(X(s-)) (\d\eta_{s}^{c} - \d\xi_{s}^{c} )  
\\
& \quad + \int_{0}^{T\wedge \beta_{n}}u'(X(s))\sigma (X(s))\d W(s) +\sum_{0\leq s \leq T\wedge \beta_{n}}\left[u(X(s)) - u(X(s-))\right].
		\end{align}
		The HJB equation \eqref{e:HJB-lta} implies that $c_{1} \le  u' (x) \le c_{2}$. Note also that $\Delta X(s)  = \Delta \xi(s) - \Delta \eta(s)$.
		Consequently, we can use  the mean value theorem to obtain \begin{equation}\label{e1:verification-proof-lta}
		u( X(s)) - u( X(s-)) \le c_{2}\Delta \xi(s) - c_{1}\Delta \eta(s).
		\end{equation} Note that \eqref{e:HJB-lta} also implies that $\LL u(x) \le -h(x) + \lambda$. Plugging this
 and \eqref{e1:verification-proof-lta} into \eqref{e0:verification-proof-lta} gives us
		\begin{align*}
		u(X(T\wedge \beta_{n})) &  \le  u(x) - \int_{0}^{T\wedge \beta_{n}}  h(X(s))\d s + \lambda (T\wedge \beta_{n})  \\& \quad  + \int_{0}^{T\wedge \beta_{n}}u'(X(s))\sigma (X(s))\d W(s) - c_{1}\eta(T\wedge \beta_{n}) + c_{2} \xi(T\wedge \beta_{n}).
		\end{align*} Rearranging
terms and then taking expectations, we arrive at
		\begin{displaymath}
		\E_{x}\bigg[ \int_{0}^{T\wedge \beta_{n}}  h(X(s))\d s +  c_{1}\eta(T\wedge \beta_{n}) - c_{2} \xi(T\wedge \beta_{n}) \bigg] + \E_{x}[u(X(T\wedge \beta_{n}))] \le  u(x) + \lambda \E_{x}[T\wedge \beta_{n}].
		\end{displaymath} Passing to the limit as $n\to\infty$ and then dividing both sides by $T$, we obtain from the nonnegativity of $u$ and the monotone convergence theorem that
		\begin{displaymath}
		\limsup_{T\to\infty}\frac1T \E_{x}\bigg[ \int_{0}^{T}  h(X(s))\d s +  c_{1}\eta(T) - c_{2} \xi(T) \bigg] \le \lambda.
		\end{displaymath} Finally, taking supremum over $\vphi\cd$ yields the assertion that $\lambda_{0} \le \lambda$.
	\end{proof}

To proceed, we make the following assumption.

\begin{assumption}\label{assumption-pi_12}
\begin{itemize}
  \item[(i)]    $\lim_{x\downarrow 0} [h(x) + c_{2} \mu(x)] \le  0$ and $\lim_{x\to \infty} [h(x) + c_{1} \mu(x)]  <   0$.

 \item[(ii)]  There exist $0 < a < b < \infty$ satisfying  \begin{equation}
\label{e:lam>0condition}
  \int_{a}^{b}h(y) m(y) \d y + \frac{c_{1}}{2s(b)} -  \frac{c_{2}}{2s(a)} > 0.
\end{equation}
\end{itemize}
\end{assumption}

 \comment{ \item[(ii)] There exists $\hat x_{i} > 0$ such that $\pi_{i} (x)$ is strictly  increasing on $(0, \hat x_{i})$ and strictly decreasing on $[\hat x_{i}, \infty) $. In addition $\lim_{x\to \infty} \pi_{1}(x) < 0$. Consequently, there exists a $b_{0} > \hat x_{1} $ such that $\pi_{1} (b_{0}) =0$.
  \item[(iii)] It holds true that
  \begin{equation}
\label{e:key-condition}
\int_{0}^{b_{0}} h(x) m(x) \d x - \frac{c_{2}}{ 2 s(0)} + \frac{c_{1}}{2 s(b_{0})}  > 0.
\end{equation}}

\begin{lemma}\label{lem:bvp neuman boundary}
Suppose Assumption \ref{assumption-pi_12} holds. Then there exists  
a positive number $\lambda= \lambda(a, b)$ so that the following boundary value problem has a solution:
\begin{equation}
\label{e:bvp-(a,b)}
\begin{cases}
   \frac12 \sigma^{2}(x) u''(x) + \mu (x) u'(x) + h(x) = \lambda, \quad x\in (a, b),\\
    u'(a) = c_{2}, \quad u'(b) = c_{1}.
    \end{cases}
\end{equation}
\end{lemma}

\begin{proof} 
 For any  $0< a < b<\infty$ and $\lambda \in \R$ given, a  solution to  the differential equation $ \frac12 \sigma^{2}(x) u''(x) + \mu (x) u'(x) + h(x) = \lambda $ is given by \begin{equation}\label{e:u-(a, b)(x)}
u(x) =   c_{1} x+ \int_{a}^{x} 2s(u) \int_{u}^{b}[c_{1} \mu(y) + h(y) - \lambda] m(y) \d y\d u, \quad  x\in [a, b].
\end{equation}  Note that $u'(b) = c_{1}$.  The  other boundary condition $u'(a) = c_{2}$ gives \begin{displaymath}
c_{1}  +2 s(a) \int_{a}^{b} [c_{1} \mu( y) + h(y) - \lambda ] m(y) \d y =c_{2}.
\end{displaymath} Note that
\begin{equation}
\label{e:convenient_identity}
\int_{a}^{b} \mu(y) m(y) \d y = \frac{1}{2s(b)} -  \frac{1}{2s(a)}, \quad \forall [a, b] \in (0, \infty).
\end{equation}
Hence it follows that
\begin{equation}\label{e1:lambda}
\lambda =\lambda(a, b) = \frac{1}{2 M[a, b]} \bigg(\frac{c_{1}}{s(b)} -  \frac{c_{2}}{s(a)} + 2 \int_{a}^{b}h(y) m(y) \d y \bigg),
\end{equation} where $M[a, b] = \int_{a}^{b} m(y) \d y > 0$ is the speed measure of the interval $[a, b]$.  In particular, it follows from \eqref{e:lam>0condition} that $\lambda$ is  positive.
\comment{ if we can choose $0< a< b< \infty$ so that $\frac{c_{1}}{s(b)} -  \frac{c_{2}}{s(a)} + 2 \int_{a}^{b}h(y) m(y) \d y  > 0$. To this end, we first observe that condition \eqref{e:h-Inada} implies that there exists a positive constant $\delta$ so that $\frac{h(y)}{y} > -2\mu c_{2}$ for all $y\in (0, \delta)$. Let $b > \delta > a$. Then  we have \begin{align*}
\frac{c_{1}}{s(b)} -  \frac{c_{2}}{s(a)} + 2 \int_{a}^{b}h(y) m(y) \d y   & \ge  -  \frac{c_{2}}{s(a)} + 2 \int_{a}^{\delta}\frac{h(y)}{y} y \frac{1}{\sigma^{2}}y^{\frac{2\mu}{\sigma^{2}}-2} \d y    \\
  & >  -  \frac{c_{2}}{s(a)} + \frac{2}{\sigma^{2}}  \int_{a}^{\delta}( -2\mu c_{2})  y^{\frac{2\mu}{\sigma^{2}}-1} \d y\\
  & =  -  \frac{c_{2}}{s(a)} + 2c_{2} \bigg( \frac{1}{s(a)} -\frac{1}{s(\delta)}\bigg) \\
  & =  \frac{c_{2}}{s(a)}- \frac{2c_{2}}{s(\delta)}.
\end{align*}\footnote{Since $2\int_{a}^{b} \mu(y)m(y)\d y= \frac{1}{s(b)} - \frac{1}{s(a)}$, we should compare $2 \int_{a}^{b}h(y) m(y) \d y$ to $2\int_{a}^{b} \mu(y)m(y)\d y$.}
Finally, we choose $a > 0$ sufficiently small so that $\frac{1}{s(a)} > \frac{2}{s(\delta)}$. For all  such choices of $a$ and $b$,  $\lambda = \lambda(a, b) > 0$. Moreover, the function $u$ of \eqref{e:u-(a, b)(x)} in which the constant $\lambda$  given by \eqref{e1:lambda}  is a solution to \eqref{e:bvp-(a,b)}. The proof is complete.}
\end{proof}

\begin{lemma}\label{lem:(LaLb)admissible}
For any $0 < a < b < \infty$,   let $X$  be the reflected diffusion process  on $[a, b]$:
\begin{equation}\label{e:reflected diffusion:a,b}
X(t) = x + \int_{0}^{t}\mu (X(s)) \d s + \int_{0}^{t}\sigma( X(s)) \d W(s) + L_{a}(t) - L_{b}(t),
\end{equation} where without loss of generality we can assume that the initial condition $x\in [a,b]$, and $L_{a}$ and $L_{b}$ denote the local time processes at $a$ and $b$, respectively.  Then there exist some positive constants $K_{1}$ and $K_{2}$ so that
\begin{displaymath}
\E_{x} [L_{a}(t) + L_{b}(t) ] \le K_{1} t + K_{2}, \quad \text{ for any } t \ge 0.
\end{displaymath} In particular, the policy $L_{a}- L_{b} \in \A_{x}$.
\end{lemma}

\begin{proof}
For the given $0 < a < b < \infty$, as in the proof of Lemma \ref{lem:bvp neuman boundary}, we can verify that the function given by
\begin{displaymath}
v(x) : = -x + a  + 2 \int_{a}^{x}s(y) \int_{y}^{b}[-\mu (z)   + K_{1}] m(z) \d z\d y,  \quad x\in [a, b]
\end{displaymath} is a solution to the boundary value problem \begin{displaymath}
\begin{cases}
   \frac12 \sigma^{2} (x) v''(x) +\mu( x) v'(x) + K_{1} =0,   &  x\in (a, b), \\
   v'(a) = 1, \ \ v'(b) = -1,
\end{cases}
\end{displaymath} where $K_{1} = K_{1}(a, b ) = \frac{1}{2 M[a,b]} ( \frac1{s(a)} +  \frac1{s(b)}) > 0.$

It is well-known that
equation  \eqref{e:reflected diffusion:a,b} has a unique solution $X$; see, for example, \cite[Section 2.4]{Harrison-85} or \cite{BurdzyKR09}.   We now apply It\^o's formula to the process $v(X(t))$,
\begin{align*}
 \E_{x}[v(X(t))] & = v(x)   + \E_{x}\bigg[ \int_{0}^{t}\bigg(\mu (X(s)) v'(X(s)) + \frac12\sigma^{2} (X (s)) v''(X(s))\bigg)\d s\bigg] \\ & \qquad \ \quad +\E_{x}\big[v'(a) L_{a}(t) - v'(b) L_{b}(t) \big]  \\
  & = v(x) - K_{1} t + \E_{x}[L_{a}(t) + L_{b}(t)].
\end{align*} Since $v$ is continuous and $X(t) \in [a, b]$ for all $t\ge 0$, it follows that there exists a positive constant  $K_{2} = K_{2}(a, b)$ such that $
\E_{x}[L_{a}(t) + L_{b}(t)] \le  K_{1} t  + K_{2}.
 $ The lemma is proved.
\end{proof}

\begin{corollary}\label{cor-lambda_0>0}
Suppose Assumption \ref{assumption-pi_12}  holds, then   $\lambda_{0} > 0 $. 
\end{corollary}

\begin{proof}
 To show that $\lambda_{0} > 0$,  we consider the function $u$ of \eqref{e:u-(a, b)(x)} in which we choose $0< a < b<\infty$ so that $\lambda = \lambda (a, b) $ of \eqref{e1:lambda} is positive. Now let $X$ be   the reflected diffusion process on $[a, b]$ given by \eqref{e:reflected diffusion:a,b}. Thanks to Lemma \ref{lem:(LaLb)admissible}, the policy $L_{a}- L_{b}\in \A_{x}$.   Apply It\^o's formula to $u(X(t))$ and then take expectations to obtain
\begin{align*}
\E_{x}[u(X(t))]  &   = u(x) + \E_{x}\bigg[ \int_{0}^{t}\bigg(\mu (X(s)) u'(X(s)) + \frac12\sigma^{2}( X(s)) u''(X(s))\bigg)\d s\bigg]\\ & \qquad \quad\ +\E_{x}\big[u'(a) L_{a}(t) - u'(b) L_{b}(t) \big] \\
  & =  u(x) + \E_{x}\bigg[ \int_{0}^{t} [\lambda-h(X(s))] \d s + c_{2}  L_{a}(t) - c_{1} L_{b}(t)\bigg].
\end{align*} Rearranging
terms,   dividing by $t$, and then passing to the limit as $t\to\infty$, we obtain
\begin{displaymath}
\lim_{t\to\infty} \frac1t\E_{x}\bigg[ \int_{0}^{t}  h(X(s)) \d s - c_{2}  L_{a}(t) + c_{1} L_{b}(t) \bigg] =\lambda = \lambda(a, b),
\end{displaymath} where $ \lambda(a, b)$ is defined in \eqref{e1:lambda}. In other words, the long-term average reward of the policy $L_{a}- L_{b}$ is $ \lambda(a, b) >0$. Now by the definition of $\lambda_{0}$, we have $\lambda_{0} \ge \lambda > 0$.
\end{proof}
		
		
\section {The   Abelian  Limit}\label{sect-abelian limit} 
 For a given $r > 0$, recall the  discounted value function $V_{r}(x)$ defined in \eqref{e-V_r(x) defn}.  Also recall the definition of $\lambda_{0}$ given in \eqref{def: lambda_0}. The following proposition presents  a relationship between the discounted and long-term average problems.

\begin{prop}\label{prop1:Abel-limit} We have $\liminf_{r\downarrow 0} r V_{r}(x)  \ge \lambda_{0}$ for any $x\in [0,\infty)$.
\end{prop}

\begin{proof} Since $V_{r}(x) \ge 0$, the relation $\liminf_{r\downarrow 0} r V_{r}(x)  \ge \lambda_{0}$ holds trivially when $\lambda_{0} =0$. Now assume that $\lambda_{0} > 0$. Let  $x\in [0,\infty)$.
For any $0 < K< \lambda_{0}$, there exists a policy $\vphi =\xi-\eta \in \A_{x}$ so that \begin{equation}
\label{e:F(T)/T>K}
\lambda_{0}\ge  \liminf_{T\to\infty} \frac1T \E_{x}\bigg[\int_{0}^{T} h(X(s))\d s - c_{2} \xi(T) + c_{1}\eta(T)  \bigg] \ge K;
\end{equation} here $X$ is the controlled process corresponding to the policy $\vphi$ with initial condition $X(0-) =x$.  Let $F(T): = \E_{x}[\int_{0}^{T} h(X(s))\d s - c_{2} \xi(T) + c_{1}\eta(T)   ] $ and $G(T) := \frac{F(T)}{T+1}$ for $T>0$.  {Thanks to \eqref{e:F(T)/T>K}, there exists a $T_{1} > 0$ such that $F(T) > 0$ for all $T \ge T_{1}$. Consequently we can use integration by parts and Fubini's theorem to obtain \begin{align*}
 \E_{x} \bigg[\int_{0}^{T} e^{-r s} (h(X(s))\d s - c_{2}\d \xi(s) + c_{1}\d\eta(s))  \bigg]
& = e^{-r T } F(T) + r\int_{0}^{T} e^{-rs} F(s) \d s\\ &  \ge  r\int_{0}^{T} e^{-rs} F(s) \d s,
\end{align*} for all $T \ge T_{1}$.}
Then it follows from the dynamic programming principle (see \cite{GuoP-05})
that for all $T \ge T_{1}$,
\begin{align*}
r V_{r}(x)  &
\ge \sup_{\vphi\in \A_{x}}  r \E_{x}  \bigg[\int_{0}^{T} e^{-r s} (h(X(s))\d s - c_{2}\d \xi(s) + c_{1}\d\eta(s)) + e^{-r T} V_{r}(X(T)) \bigg]   \\
&  \ge r \E_{x} \bigg[\int_{0}^{T} e^{-r s} (h(X(s))\d s - c_{2}\d \xi(s) + c_{1}\d\eta(s))  \bigg]     \\
&    \ge {   r^{2} \int_{0}^{\infty} e^{-r s} F(s) \d s} = r^{2} \int_{0}^{\infty} e^{-r s} (s+1) G(s) \d s = \int_{0}^{\infty} e^{-t}(t+r) G(t/r) \d t.
\end{align*} In view of  \eqref{e:F(T)/T>K}, for any $\e > 0$,  we can find  a $T_{2}>0$ so that $G(T) = \frac{F(T)}{T+1} \ge K-\e$ for all $T\ge T_{2}$. Denote $T_{0} : = T_{1}\vee T_{2}$. Then we can estimate
\begin{align*}
 r V_{r}(x)  & \ge  \int_{0}^{ r T_{0}} e^{-t}(t+r) G(t/r) \d t + \int_{r T_{0}}^{\infty} e^{-t}(t+r) G(t/r) \d t\\ &\ge  \int_{0}^{ r T_{0}} e^{-t} F(t/r) \d t +   \int_{r T_{0}}^{\infty} e^{-t}(t+r) (K-\e ) \d t\\
  & \ge -r^{2} T_{0} \max_{t\in [0, T_{0}]} F(t) + (K-\e) (1+r -r^{2}(T_{0}+1)).
\end{align*} Letting $r\downarrow 0$, we obtain $\liminf_{r\downarrow 0} r V_{r}(x)  \ge K -{\e}$. Since $\e > 0$ and $K < \lambda_{0} $ are arbitrary, we obtain $\liminf_{r\downarrow 0} r V_{r}(x)  \ge \lambda_{0}$, which completes  the proof.
\end{proof}
Usually one can show that the value function $V_{r}$ of  \eqref{e-V_r(x) defn} is a viscosity solution to the HJB equation
\begin{eqnarray}\label{e:HJB-r}
		\min\{rv(x)  -\LL v(x) -h(x), -v'(x) + c_2, v'(x)-c_1\} = 0, \quad x\in(0,\infty).
		\end{eqnarray}
Under additional assumptions such as concavity of the function $h$, for specific models (such as geometric Brownian motion in \cite{GuoP-05}), one can further show that   $V_{r}$ is a smooth solution to \eqref{e:HJB-r} and that
 there exist $0 <  a_{r} < b_{r} < \infty$ so that 
 the reflected diffusion process on the interval $[a_{r}, b_{r}]$ is an optimal state process. In other words, the policy $L_{a_{r}}- L_{b_{r}}$ is an optimal control policy, where $L_{a_{r}}$ and $ L_{b_{r}}$ denote  the local times of the controlled process $X$ at $a_{r}$ and $b_{r}$, respectively. If the initial position $X(0-)$ is outside the interval $[a_{r}, b_{r}]$, then an initial jump to the nearest boundary point is exerted at time 0.  
 We also refer to \cite{Mato-12} and \cite{Weera-05} for sufficient conditions for the optimality of such policies for general one-dimensional diffusion processes under different settings.

 Motivated by these recent developments,
we make the following assumption:
\begin{assumption}\label{assumption-V_r}
For each $r> 0$,   there exist two numbers with $0 <  a_{r} < b_{r} < \infty$ so that  the discounted value function $V_{r}(x)$  of  \eqref{e-V_r(x) defn}  is $C^{2}([0,\infty))$ and satisfies the following   system of equations:
\begin{equation}\label{e2:HJB-r}
\begin{cases}
  r  V_{r}(x)  -\LL  V_{r}(x)  -h (x) =0,  \quad  \  c_{1} \le V_{r}'(x) \le c_{2},     & x\in (a_{r}, b_{r}), \\
 r  V_{r}(x)  -\LL  V_{r}(x)  -h (x)  \ge 0, \quad \ V_{r}'(x) = c_{1},  & x\ge b_{r},\\ 
   r  V_{r}(x)  -\LL  V_{r}(x)  -h (x) \ge 0,  \quad\  V_{r}'(x) = c_{2},& x\le a_{r}. 
\end{cases}
\end{equation}
\end{assumption}

In addition, the following assumption is needed for the proof of Theorem  \ref{thm-lta-main result}.
\begin{assumption}\label{assumption-mu-sigma}
The functions $h$, $\mu$,  and $\sigma$ are continuously differentiable and satisfies $\inf_{x\in [a, b]} \sigma^{2}(x) > 0$ for any $[a, b] \subset (0, \infty)$.
\end{assumption}

We now state the main result of this section.

\begin{theorem}\label{thm-lta-main result}
Let  Assumptions  \ref{assumption-pi_12}, \ref{assumption-V_r}, and  \ref{assumption-mu-sigma}  hold. Then 
there exist positive constants
$ a_{*} < b_{*}$ so that  the following  statements hold true:
\begin{itemize}
  \item[(i)] $\lim_{r\downarrow 0}rV_r(x) = \lambda_{0}$ for all $x \in \mathbb{R}$.
  \item[(ii)]  The reflected diffusion process on the state space $[a_{*},  b_{*}]$ $($if the initial point is outside this interval, then there will be an initial jump to the nearest point of the interval$)$ is an optimal state process for the ergodic control problem \eqref{def: lambda_0}. Hence the optimal control policy here is given by $\vphi_{*} = L_{a_{*}} - L_{b_{*}}$, in which $L_{a_{*}} $ and $L_{b_{*}}$ denote  the local time processes at $a_{*}$ and $b_{*}$, respectively.
\end{itemize}
\end{theorem}

To prove Theorem \ref{thm-lta-main result}, we first establish a series of technical lemmas.

 \begin{lemma}\label{lem-ar>0} Suppose Assumptions \ref{assumption-pi_12} (i) and \ref{assumption-V_r} hold. Then
 $\lambda_{0} > 0$  if and only if  $\liminf_{r\downarrow 0} a_{r} >0$.
\end{lemma}

\begin{proof}
Recall that the function   $V_{r}\in C^{2}(0, \infty)$ satisfies $r V_{r}(x) - \mu( x) V_{r}'(x) - \frac12 \sigma^{2}(x) V''_{r} (x) - h(x) = 0$ for $x\in (a_{r}, b_{r})$ and $V_{r}(x) = c_{2} x + V_{r}(0+)$ for $x\le a_{r}$, where $V_{r}(0+)= \lim_{x\downarrow 0} V_{r}(x) = V_{r}(a_{r}) - c_{2} a_{r}$.   
Therefore
 the smooth pasting principle for $V_{r}$ at $a_{r}$ implies that $V_{r}'(a_{r}) = c_{2}$ and $V_{r}''(a_{r}) =0$. Thus it follows that  \begin{equation}
\label{e:6.13GP}
r(c_{2}a_{r} + V_{r}(0+)) -c_{2}\mu(a_{r}) - h(a_{r}) =0 \quad \text{ or } \quad  rV_{r}(0+) = h(a_{r}) +c_{2}\mu(a_{r})- r c_{2}a_{r}.  
\end{equation} 

(i) Suppose first that $\lambda_{0} > 0$. Then Proposition \ref{prop1:Abel-limit} implies that any limit point of $\{r V_{r}(0+)\}$ must be greater than 0.  If there exists a sequence $\{r_{n}\} \subset (0, 1]$ for which $\lim_{n\to\infty} a_{r_{n}} =0$, then passing to the limit as $n\to\infty$ in \eqref{e:6.13GP} will give us $\lim_{n\to\infty} r_{n} V_{r_{n}}(0+) = \lim_{n\to\infty}[ h(a_{r_{n}}) + c_{2} \mu(a_{r_{n}}) ] \le  0$ thanks to Assumption \ref{assumption-pi_12} (i).    This is a contradiction. 

(ii)  Now if $\liminf_{r\downarrow 0} a_{r} =0$, then using \eqref{e:6.13GP}  again, we obtain  $\lim_{n\to\infty} r_{n} V_{r_{n}}(0+) = 0$ for some sequence  $\{r_{n}\}$ such that $\lim_{n\to\infty} r_{n } =0$ and
   $\lim_{{n }\to \infty} a_{r_{n}} =0$. 
This, together with Proposition \ref{prop1:Abel-limit},
indicates that $\lambda_{0} =0$.
\end{proof}

\begin{lemma}\label{lem-br<infty}
Suppose Assumptions \ref{assumption-pi_12} (i) and \ref{assumption-V_r} hold. Then   $\limsup_{r\downarrow 0 } b_{r}  < \infty$.
\end{lemma}

\begin{proof}
As in the proof of Lemma \ref{lem-ar>0}, we can use the smooth pasting
for the function $V_{r}$ at $b_{r}$ to obtain \begin{equation}\label{e:6.14GP}
0 < r V_{r} (b_{r}) = c_{1} \mu (b_{r}) + h(b_{r}).
\end{equation} 
Suppose that there exists some sequence $\{r_{n}\} \subset (0, 1]$ so that $\lim_{n\to\infty} r_{n} =0$ and $\lim_{n\to\infty} b_{r_{n}} =\infty$. Now passing to the limit as $n\to\infty$ in  \eqref{e:6.14GP}, we obtain   $$ 0\le \lim_{n\to\infty}  r_{n} V_{r_{n}} (b_{r_{n}}) =  \lim_{n\to\infty} ( c_{1} \mu( b_{r_{n}}) + h(b_{r_{n}})) < 0$$
 thanks to Assumption \ref{assumption-pi_12} (i).  This is a contradiction. Hence $\limsup_{r\downarrow 0 } b_{r}  < \infty$ and the assertion of the lemma follows. \end{proof}


The following lemma can be obtained   directly from Corollary \ref{cor-lambda_0>0} and Lemmas \ref{lem-ar>0} and \ref{lem-br<infty}.

\begin{lemma}\label{lem-ar-br-bounded}
Suppose Assumptions \ref{assumption-pi_12} (i) and \ref{assumption-V_r} hold.  Then  there exist     constants $0 <r_{0} <1$ and $0 < K_{1} < K_{2} < \infty$ so that $  K_{1} \le a_{r} < b_{r} \le K_{2} $ for all  $0 < r \le r_{0}$.
\end{lemma}

\begin{lemma}\label{Bounded Boundaries}
Let  Assumptions  \ref{assumption-pi_12}, \ref{assumption-V_r}, and  \ref{assumption-mu-sigma}  hold. Then  there exists a function $w_{r}\in C^{1}((0,\infty)) \cap C^{2}((0,\infty)\setminus \{a_{r}, b_{r}\})$ satisfying \begin{align}\label{e:w_r-equation}\begin{cases}
     \frac12 \sigma^{2}( x) w_{r}''(x) + (\sigma (x) \sigma'(x) + \mu(x) )   w_{r}'(x) - (r - \mu'(x) ) w_{r}(x) + h'(x) = 0, &  x\in (a_{r}, b_{r}); \\
     c_{1} \le w_{r}(x) \le  c_{2},  & x\in (a_{r}, b_{r}), \\
      w_{r} (x) = c_{2},  \ \   w'_{r} (x) = 0, &  x\le a_{r},
      \\  w_{r}(x) = c_{1}, \ \ w'_{r} (x) = 0, & x\ge b_{r}.
\end{cases}
\end{align}
\end{lemma}

\begin{proof} Recall  Assumption \ref{assumption-V_r} indicates
that  $V_{r}$ of   \eqref{e-V_r(x) defn} satisfies
\begin{align}\label{e:v_r-ode}
\frac12 \sigma^{2} (x) V_{r}''(x) + \mu( x) V_{r}'(x) - r V_{r}(x) + h(x) = 0, \quad x\in (a_{r}, b_{r}),
\end{align} \begin{displaymath}
 c_{1} < V_{r}'(x) <  c_{2}, \quad x\in (a_{r}, b_{r}),
\end{displaymath}
and \begin{align*}
 V_{r}'(x) = c_{2},  \text{ for } x\le a_{r},  \ \     V_{r}'(x) = c_{1},   \text{ for }x\ge b_{r},  \text{ and }V_{r}''(x) =0 \text{ for }x \le a_{r} \text{ or }x \ge b_{r}.
\end{align*}  Now  differentiating \eqref{e:v_r-ode} and denoting   $w_{r}: = V'_{r}$,
then
$w_{r}$ satisfies \eqref{e:w_r-equation}.
\end{proof}

\begin{lemma}\label{lem-w_0}
Let  Assumptions  \ref{assumption-pi_12}, \ref{assumption-V_r}, and  \ref{assumption-mu-sigma}  hold.  Then there exist two positive constants $a_{* } < b_{*}$,  a constant $l_{0}>0$, and a function $w_{0}\in C^{1}(0,\infty)$ satisfying  \begin{equation}\label{e:w_0-equats}
 \begin{cases}
    \frac12 \sigma^{2} (x) w_{0}'(x) +  \mu (x) w_{0}(x) + h(x) = l_0, & \quad  x\in (a_{*}, b_{*}),\\
 c_{1} \le  w_{0}(x) \le  c_{2}, & \quad  x\in (a_{*},  b_{*}),\\
w_{0} (x) = c_{2},  \ \   w'_{0} (x) = 0, & \quad  x\le a_{*},\\
 w_{0}(x) = c_{1}, \ \ w'_{0} (x) = 0, & \quad x\ge b_{*}.
\end{cases}
\end{equation}
\end{lemma}

\begin{proof} Thanks to \eqref{e:w_r-equation},   $w_{r}$ is uniformly bounded.
Next we  integrate the first equation of \eqref{e:w_r-equation} from $a_{r}$ to $x$ ($x\in (a_{r}, b_{r})$) to obtain
\begin{equation}\label{e:w_r-integral-eq}
\frac12  \sigma^{2} (x) w_{r}'(x) =  c_{2}\mu (a_{r}) + h(a_{r})  - \mu (x) w_{r}(x) + r \int_{a_{r}}^{x} w_{r}(y) \d y - h(x).
\end{equation}
Thanks to Lemma \ref{lem-ar-br-bounded}  and the continuity of $h$ and $\mu$,
the right-hand side of \eqref{e:w_r-integral-eq} is uniformly bounded on $[K_{1}, K_{2}]$  for all $r\in (0, r_{0}]$, where $r_{0}, K_{1}$, and $ K_{2}$ are the positive constants found in Lemma \ref{lem-ar-br-bounded}. This, together with Assumption \ref{assumption-mu-sigma}, implies that   $\{w_{r}'(x), r\in (0, r_{0}] \}$  is uniformly bounded on $[K_{1}, K_{2}]$. Consequently,   $\{w_{r}(x), r\in (0, r_{0}] \}$ is equicontinuous.

 Now we rewrite   the first equation of \eqref{e:w_r-equation} as
\begin{displaymath}
  \frac12 \sigma^{2} (x) w_{r}''(x) = -(\sigma (x) \sigma'(x) + \mu(x) )   w_{r}'(x) + (r - \mu'(x) ) w_{r}(x) - h'(x),  
   \quad x\in (a_{r}, b_{r}).
\end{displaymath} Since $\{w_{r}'\}$ and $\{w_{r}\}$ are uniformly bounded and $\sigma, \sigma'$, $\mu'$, and  $h'$ are  continuous, it follows from Assumption \ref{assumption-mu-sigma} that there exists a positive constant $K$ {\em independent of} $r$ such that
 \begin{displaymath}
|w_{r}''(x)| \le K, \text{ for all }x\in (a_{r}, b_{r}) \subset [K_{1}, K_{2}], \text{ for any } r\in (0, r_{0}].
\end{displaymath} This, together with the fact that $w'_{r}(x)  =0$ for all $x\in [K_{1}, K_{2}]\setminus (a_{r}, b_{r})$ (c.f. \eqref{e:w_r-equation}), implies that 
  $\{w_{r}', r\in (0, r_{0}]  \}$ is equicontinuous on $[K_{1}, K_{2}]$.
  \comment{For any $\e > 0$, let $\delta: = \frac{\e}{K}$. For  any $ r\in (0, r_{0}] $ and  $x< y \in [K_{1}, K_{2}]$ with $|x-y| < \delta$, we show that $|w_{r}'(x) - w_{r}'(y) |  < \e$ in  the following  cases:
  \begin{itemize}
  \item[(i)] $x< y \le a_{r}$. In this case, by \eqref{e:w_r-equation}, $|w_{r}'(x) - w_{r}'(y) | = | 0 - 0| < \e. $
  \item[(ii)] $b_{r} \le x < y$.  This is similar to case (i) since $w_{r}'(x) = w_{r}'(y) = 0$.
  \item[(iii)] $x\le a_{r} < y < b_{r}$. In this case, using \eqref{e:w_r-equation},  we have $|w_{r}'(x) - w_{r}'(y) | = |w_{r}'(a_{r}) - w_{r}'(y) | \le K|y-a_{r}| \le K|y-x| <\e$.
   \item[(iv)] $x\le a_{r} < b_{r} \le y$. Similar to the previous case,  we  use \eqref{e:w_r-equation} to derive $|w_{r}'(x) - w_{r}'(y) | = |w_{r}'(a_{r}) - w_{r}'(b_{r}) | \le K|b_{r}-a_{r}| \le K|y-x| <\e$.
   \item[(v)] $a_{r} < x < y < b_{r}$. Again, it is obvious that   $|w_{r}'(x) - w_{r}'(y) | \le  K|y-x| <\e$.
   \item[(vi)] $a_{r} < x  < b_{r} \le y$. As before, we have  $|w_{r}'(x) - w_{r}'(y) | =  |w_{r}'(x) - w_{r}'(b_{r}) | \le K|b_{r} -x| \le K|x-y| < \e$.
\end{itemize}
Combining the cases above, we arrive at the desired assertion that  $\{w_{r}', r\in (0, r_{0}]  \}$ is equicontinuous on $[K_{1}, K_{2}]$.}

Meanwhile,  Lemma \ref{lem-ar-br-bounded}    implies that there exists a sequence $\{r_{n}\}_{n\geq 1}$ such that $\lim_{n\to \infty}r_{n} = 0$ for which
\begin{equation}\label{e:a_rb_r-limit}
\lim\limits_{n\to \infty}a_{r_{n}} = a_{*}, \ \quad \
\lim\limits_{n\to \infty}b_{r_{n}} = b_{*}
\end{equation}
for some $0<K_1 \leq a_{*} \leq b_{*} \leq K_2 < \infty$.  Since $\mu$ and $h$ are continuous, we have $$\lim_{n\to\infty}( c_{2}\mu ( a_{r_{n}}) + h(a_{r_{n}}) ) =c_{2} \mu(  a_{*}) + h(a_{*}) =: l_{0}.$$ Recall that
\eqref{e:6.13GP} indicates
$r V_{r}(0+) = h(a_{r}) + c_{2} \mu (a_{r})-r  c_{2} a_{r}$. 
Thus $l_{0}$ is a limiting point of  $\{rV_r(0+)\}$. This, together with Proposition \ref{prop1:Abel-limit}, implies that $l_{0}\ge \lambda_{0}$.

Since $\{w_{r_{n}}\}$ and $\{w'_{r_{n}}\}$ are equicontinuous and uniformly bounded  sequences  on $[K_1, K_2]$,  by the Arzel\`{a}-Ascoli theorem, there exists a continuously differentiable function $w_0$ on $[K_1,K_2]$ such that for some subsequence of $\{r_{n}\}$ (still denoted by $\{r_{n}\}$), we have
\begin{eqnarray}\label{e:w_r-limit}
\lim\limits_{n\to \infty}w_{r_{n}}(x) = w_0(x)\ \quad \
\lim\limits_{n\to \infty}w'_{r_{n}}(x) = w_0'(x).
\end{eqnarray}
Passing to the limit along the subsequence $\{r_{n}\}$  in \eqref{e:w_r-integral-eq} and noting that $w_{r_{n}}$ is uniformly bounded, we obtain from \eqref{e:a_rb_r-limit} and \eqref{e:w_r-limit} that
\begin{displaymath}
\frac12 \sigma^{2}(x) w_{0}'(x) + \mu (x) w_{0}(x)   + h(x) =l_{0}, \quad x\in (a_{*},b_{*}).
\end{displaymath} Since $\{w_{r_{n}}\}$ and $\{w'_{r_{n}}\}$ are equicontinuous, we can extend $w_{0}$ to $(0,\infty)$ so that $w_{0}$ is continuously differentiable and that the third and fourth lines of \eqref{e:w_0-equats} are satisfied. That $ c_{1} \le  w_{0}(x) \le  c_{2} $ for all $  x\in (a_{*},  b_{*})$ is obvious. This completes the proof of the lemma.
\end{proof}

We are now ready to prove  Theorem \ref{thm-lta-main result}.

\begin{proof}[Proof of Theorem \ref{thm-lta-main result}]
Let the positive constants $l_{0}$ and $a_{*} < b_{*}$ and the function $w_{0}$ be as in the statement of Lemma \ref{lem-w_0}. Define the function $Q$ by
\begin{eqnarray*}
Q(x) = \int_{0}^{x}w_{0}(u)\d u, \quad x\in (0,\infty).
\end{eqnarray*} Since $w_{0}$ is positive and satisfies \eqref{e:w_0-equats}, it is easy to see that $Q$ is nonnegative and satisfies
\begin{equation}
\label{e:Q(x)qvi}
\begin{cases}
   \frac12 \sigma^{2}( x) Q''(x) + \mu (x) Q'(x) + h(x) = l_{0}, \  \ & x\in (a_{*},b_{*}),    \\
  Q'(x) = c_{2}, \ \     &  x\le a_{*},\\
  Q'(x) = c_{1}, \ \     &  x\ge b_{*}.
\end{cases}
\end{equation}

 Without loss of generality, we assume that $x$ is in the interval $[a_{*},b_{*}]$. Let $X_{*}$ be the reflected diffusion process on the interval $[a_{*}, b_{*}]$ with $X_{*}(0) =x$; that is,
 \begin{displaymath}
X_{*}(t) = x + \int_{0}^{t} \mu (X_{*}(s)) \d s + \int_{0}^{t} \sigma( X_{*}(s)) \d W(s) + L_{a_{*}}(t) - L_{b_{*}}(t),
\end{displaymath}  where $L_{a_{*}} $ and $L_{b_{*}}$ denote the local time processes at $a_{*}$ and $b_{*}$, respectively. Thanks to Lemma \ref{lem:(LaLb)admissible},  $L_{a_{*}}-L_{b_{*}} \in \A_{x}$ and $X_{*}$ is an admissible process.
\comment{Denote by $L_{a^*}$ and $L_{b^*}$ the local-time processes of $X$ at the points $a^*$ and $b^*$, respectively.  Consider the reflected diffusion process $X$ on the interval $[a^*,b^*]$ given by
\begin{eqnarray}
X(t) = x + \int_{0}^{t}\mu(X(s))ds + \int_{0}^{t}\sigma(X(s))ds + L_{a^*}(t) - L_{b^*}(t)
\end{eqnarray}}
We now apply  It\^o's formula to $Q(X_{*}(T))$ and use \eqref{e:Q(x)qvi}  to obtain 
\begin{align*}
\mathbb{E}_{x}\left[Q(X_{*}(T))\right] &= Q(x) + l_{0} T + \mathbb{E}_{x}\bigg[-\int_{0}^{T} h (X_{*}(s))\d s + c_{2}  L_{a_{*}}(t) -c_{1} L_{b_{*}}(t) \bigg].
\end{align*}	
Rearranging
terms,   dividing both sides by $T$, and then letting $T\to\infty$, we obtain \begin{displaymath}
\lim_{T \to \infty}\frac{1}{T}\mathbb{E}_{x}\bigg[\int_{0}^{T}h(X(s))ds + c_1L_{b_{*}}(T) - c_2L_{a _{*}}(T) \bigg] =  l_0.
\end{displaymath}
This implies that $l_{0}\le \lambda_{0}$. Recall we have observed in the proof of Lemma \ref{lem-w_0}  that $l_{0}\ge \lambda_{0}$. Hence we conclude that  $l_{0}= \lambda_{0}$ and
$\vphi_{*} = L_{a_{*}} - L_{b_{*}}$ is an optimal policy.

Recall from the proof of Lemma \ref{lem-w_0} that $l_{0}$ is a limiting point of $\{r V_{r}(0+)\}$. Since $\lambda_{0} = l_{0}$, it follows that any limiting point of $\{r V_{r}(0+)\}$ is equal to $\lambda_{0}$. Moreover, for any $x \in (0,\infty)$, we have \begin{align*}
 \lim_{r\downarrow 0} r V_{r}(x) & = \lim_{r\downarrow0} r (V_{r}(x) - V_{r}(0+)) + \lim_{r\downarrow 0} r V_{r}(0+) \\&  = \lim_{r\downarrow0} r w_{r} (\theta x) x + \lambda_{0}     \\
  & = \lambda_{0},
\end{align*}  where $\theta \in [0,1]$. Note that we used the fact that $w_{r}$ is uniformly bounded to obtain the last equality.  The proof is complete.
\end{proof}

\section{The Ces\`aro Limit}\label{sect-cesaro limit}
	
This section establishes the Ces\`aro limit $\lim_{T\to \infty}\frac{1}{T}V_{T}(x) = \lambda_{0}$, where $V_{T}(x)$ is the value function for the finite-horizon problem defined in \eqref{e-V_T(x) defn}.

\begin{theorem}\label{thm-cesaro-limit}
	 Let  Assumptions  \ref{assumption-pi_12}, \ref{assumption-V_r}, and  \ref{assumption-mu-sigma}  hold. Then 	we have
			\begin{equation}\label{e:1/T V_T -> lambda_0}
			\lim_{T\to \infty}\frac{1}{T}V_{T}(x) = \lambda_{0}.
			\end{equation}
		\end{theorem}
		
\begin{proof} The proof is divided into two steps.

{\em Step 1.} We first prove \begin{equation}\label{liminf 1/TV_T(x) < lambda_0}
				\liminf\limits_{T\to \infty}\frac{1}{T} V_T(x) \geq \lambda_{0}.
				\end{equation}
Note that \eqref{liminf 1/TV_T(x) < lambda_0} is obviously true when $\lambda_{0} =0$. Now we consider the case when $\lambda_{0} > 0$ and let $K$ be a constant so that $K <\lambda_{0}$. By \eqref{def: lambda_0} there exits an admissible control $\vphi:=(\xi,\eta)\in \A_{x}$ so that
				\begin{equation*}\label{limsup 1/TE[] > K}
				\liminf_{T\to\infty} \frac{1}{T}\mathbb{E}_{x}\left[\int_{0}^{T} h(X(s))\d s +c_1\eta(T) - c_2\xi(T)\right] > K.
				\end{equation*}
	For any $T > 0$, 	by the definition of $V_{T}(x)$, we have $\mathbb{E}_{x}[\int_{0}^{T} h(X(s)) \d s +c_1\eta(T) - c_2\xi(T)]  \le V_{T}(x)$ and hence
				\begin{align*}
					K < \liminf_{T\to\infty} \frac{1}{T}\mathbb{E}_{x}\left[\int_{0}^{T} h(X(s)) \d s +c_1\eta(T) - c_2\xi(T)\right]
					\leq \liminf_{T\to\infty} \frac{1}{T}V_{T}(x).
				\end{align*}	Since $K < \lambda_{0}$ is arbitrary we conclude that $\lambda_{0} \leq \liminf_{T\to\infty} \frac{1}{T}V_{T}(x)$. This gives \eqref{liminf 1/TV_T(x) < lambda_0}.

{\em Step 2.}
We now show that \begin{equation}
 \label{limsup V_T(x)/T > lambda_0}
\limsup_{T\to\infty} \frac{1}{T}V_{T}(x) \le  \lambda_{0}.
\end{equation} To this end,  consider   any $\vphi \in \A_{x}$ for which $\E_{x}[\int_{0}^{T} h(X(s)) \d s + c_{1} \eta(T) - c_{2}\xi(T)] \ge (V_{T}(x) -1)\vee 0$, where  $X$ is the corresponding controlled process with $X(0-) =x \in (0, \infty)$.

Let $r> 0$ and consider the functions $V_{r}$ and   $w_{r}$  defined respectively in  \eqref{e-V_r(x) defn} and \eqref{e:w_r-equation}. Recall that $w_{r} = V_{r}'\in [c_{1}, c_{2}]$. Hence for any $x> 0$, we can write
\begin{align}\label{e:V_r-ineq}
V_{r}(x) = V_{r}(0) + \int_{0}^{x} w_{r}(y) \d y \ge  V_{r}(0) + c_{1} x.
\end{align}
We now apply It\^o's formula to the process $e^{-r T} V_{r}(X(T))$ to obtain
\begin{align*}
\E&_{x} [e^{-r(T\wedge \beta_{n})} V_{r}(X(T\wedge \beta_{n}))]  \\& = V_{r} (x)  + \E_{x}\bigg[ \int_{0}^{T\wedge \beta_{n}} e^{-rs} (-r V_{r} + \LL V_{r})  (X(s)) \d s + e^{-rs} V_{r}' (X(s)) (\d \xi^{c} (s) - \d \eta^{c} (s))\bigg]  \\
  & \qquad + \E_{x}\Bigg[\sum_{0\leq s \leq T\wedge \beta_{n}} e^{-rs}\left[V_{r} (X(s)) -  V_{r} (X(s-))\right]\Bigg],
\end{align*}	where $\beta_{n}: = \inf\{t\ge 0: X(t) \ge n\}$ and $n \in \N$.	Since $V_{r}$ satisfies the HJB equation \eqref{e:HJB-r}, we have
\begin{align*}
\E&_{x} [e^{-r(T\wedge \beta_{n})} V_{r}(X(T\wedge \beta_{n}))]     \le V_{r} (x)  + \E_{x}\bigg[ \int_{0}^{T\wedge \beta_{n}} e^{-rs} [-h(X(s)) \d s+ c_{2}\d\xi(s) - c_{1} \d\eta(s) ]\bigg].
\end{align*}
	Rearranging
terms and using \eqref{e:V_r-ineq} yield
\begin{align*}
 \E&_{x}\bigg[ \int_{0}^{T\wedge \beta_{n}} e^{-rs} [h(X(s))\d s- c_{2}\d\xi(s) + c_{1} \d\eta(s)] \bigg]\\ &  \le  V_{r} (x) - \E_{x} [e^{-r(T\wedge \beta_{n})} V_{r}(X(T\wedge \beta_{n}))]  \\&  \le   V_{r} (x)  -  \E_{x} [e^{-r (T\wedge \beta_{n})} V_{r}(0) +e^{-r(T\wedge \beta_{n})} c_{1} X(T\wedge \beta_{n})]\\
 & \le  V_{r} (x)  -  \E_{x} [e^{-r (T\wedge \beta_{n})} V_{r}(0) ]. 
\end{align*} Then it follows that 	
\begin{align}\label{e4.29}
 \E_{x} \bigg[ \int_{0}^{T } e^{-rs} (h(X(s))\d s - c_{2}\d\xi(s) + c_{1}\d \eta(s)) \bigg]    \le    V_{r} (x)  -   e^{-r T} V_{r}(0). 
\end{align} 
	Using integration by parts, we have
\begin{align*} &  e^{-r T} \E_{x}\bigg[\int_{0}^{T } h(X(s))\d s - c_{2}\xi(T) + c_{1}  \eta(T)\bigg]\\& \ \  =  \E_{x} \bigg[ \int_{0}^{T } e^{-rs} (h(X(s))\d s - c_{2}\d\xi(s) + c_{1}\d \eta(s)) \bigg]   \\  &\quad \ \
  - r\E_{x}\bigg[\int_{0}^{T}\!\! e^{-r t}\bigg[ \int_{0}^{t} h(X(s))\d s - c_{2}\xi(t) + c_{1}  \eta( t)\bigg]\d t\bigg].
\end{align*}
Plugging this observation into \eqref{e4.29}, we obtain
\begin{align*}
& e^{-r T} \E_{x}\bigg[\int_{0}^{T } h(X(s))\d s - c_{2}\xi(T) + c_{1}  \eta(T)\bigg] \\ &  \   \le  V_{r} (x)  -   e^{-r T} V_{r}(0)  
-  r\E_{x}\bigg[\int_{0}^{T}e^{-r t} \bigg[\int_{0}^{t} h(X(s))\d s - c_{2}\xi(t) + c_{1}  \eta( t)\bigg]\d t\bigg]
   \\ &   \  \le  V_{r} (x)  -   e^{-r T} V_{r}(0) + r c_{2} \int_{0}^{T}e^{-r t} \E[\xi (t) ] \d t\\
   & \ \le  V_{r} (x)  -   e^{-r T} V_{r}(0) + r c_{2} \int_{0}^{T}e^{-r t}  ( K_{1}(x) t^{n} +K_{2}(x)  )\d t\\
   & \ =  V_{r} (x)  -   e^{-r T} V_{r}(0) + \frac{c_{2}K_{1}(x) n!}{r^{n}}\bigg ( 1 - e^{-rT}\sum_{j=0}^{n} \frac{ (rT)^{j}}{j!} \bigg) + K_{2}(x)(1- e^{-r T})\\
   & \ = (1-   e^{-r T}) V_{r}(0) +  [ V_{r} (x) - V_{r}(0)]  +  \frac{c_{2}K_{1}(x)n!}{r^{n}}\bigg ( 1 - e^{-rT}\sum_{j=0}^{n} \frac{ (rT)^{j}}{j!} \bigg) + K_{2}(x)(1- e^{-r T}),
\end{align*} where the third last step follows from \eqref{e:set A defn}.
  Now we pick $r = \frac{\delta}{T}$ for some $\delta \in (0,1)$ and divide both sides by $T$ to obtain
  \begin{align*}
& e^{-\delta}  \frac1T\E_{x}\bigg[\int_{0}^{T } h(X(s))\d s - c_{2}\xi(T) + c_{1}  \eta(T)\bigg] \\& \le \frac{1-   e^{- \delta}}{\delta} r V_{r}(0) + \frac{V_{r} (x)  -   e^{-\delta} V_{r}(0) }{T} 
 + \frac{c_{2}K_{1}(x)n! T^{n-1}}{\delta^{n}} \bigg ( 1 - e^{-\delta}\sum_{j=0}^{n} \frac{ \delta^{j}}{j!} \bigg) + \frac{K_{2}(x)(1- e^{- \delta})}{T}.
\end{align*}
Since $\lim_{r\downarrow 0}  r V_{r}(0) = \lambda_{0}$ thanks to  Theorem \ref{thm-lta-main result}, we first  let $\delta \downarrow 0$ and then let $T \to \infty$ to obtain
\begin{displaymath}
\limsup_{T\to \infty} \frac1T\E_{x}\bigg[\int_{0}^{T } h(X(s))\d s - c_{2}\xi(T) + c_{1}  \eta(T)\bigg] \le  \lambda_{0}.
\end{displaymath} This is true for any $\vphi \in \A_{x}$ and hence it follows that $\limsup_{T\to \infty} \frac{V_{T}(x)}{ T} \le \lambda_{0}$; establishing \eqref{limsup V_T(x)/T > lambda_0}. The proof is complete.
\end{proof}

	\section{Geometric Brownian Motion Models}\label{sect-gbm}
	 In this section, we apply the vanishing discount method to study   a two-sided singular control problem when the underlying process is a geometric Brownian motion. That is, we consider	 the following long-term average  singular control problem
	 \begin{equation}
	\label{def:lambda_0-gbm} \begin{cases}
	 \lambda_{0} : =\displaystyle\sup_{\vphi(\cdot)\in \A_{x}} \liminf_{T\to\infty}\frac1T\E_{x}\bigg[\int_{0}^{T} h(X(s) ) \d s +c_{1} \eta(T)- c_{2}    \xi(T) \bigg],\\[2ex]
	 \text{subject to }\begin{cases}\d X(s)= \mu X(s)\d s +\sigma X(s)\d W(s) +\d\xi(s) - \d\eta(s), \quad s\ge 0,\\
	X(0-)=x \ge 0, \end{cases}
 \end{cases}
	\end{equation}
	 where $\mu < 0$,  $\sigma > 0$,  $0 < c_{1} < c_{2}$ are   constants, $h$ is a nonnegative function satisfying Assumption \ref{h-assumption} below, and the singular control $\vphi = \xi -\eta$ is admissible as in Section \ref{sect-formulation}.
	
	  In problem \eqref{def:lambda_0-gbm}, the underlying uncontrolled process is a geometric Brownian motion with state space $(0, \infty)$. Using the criteria in Chapter 15 of \cite{KarlinT81}, we see that both 0 and $\infty$ are natural boundary points. Moreover, the scale and speed densities are     respectively given by $s(x)=x^{-\frac{2\mu}{\sigma^{2}}} $ and $m(x)= \frac{1}{\sigma^{2}} x^{\frac{2\mu}{\sigma^{2}}-2}$, $x> 0$.

\begin{rem}
Let us explain why we need to assume $\mu < 0$ in \eqref{def:lambda_0-gbm}.
	Suppose  that $\mu > 0$, then one can show that $\lambda_{0}=\infty$. Indeed, for any $x > 0$  and any sequence $t_{k}\to\infty$ as $k\to\infty$, we can construct an admissible policy $\vphi\cd=(\xi\cd,\eta\cd)$ as follows. Let $\xi(t)\equiv 0$ for all $t\ge 0$,   $\eta(t) = 0 $ for $t< t_{k}$ and $\eta(t) \equiv \Delta \eta(t_{k}) = X_{0}(t_{k}):=  x \exp\{(\mu -\frac12 \sigma^{2})t_{k} + \sigma W(t_{k}) \}$ for all $t\ge t_{k}$.   In other words, under the policy $\vphi\cd$, the manager does nothing before  time $t_{k}$ and then harvests all at time $t_{k}$.   Since  $X_{0}(t_{k}) = x \exp\{(\mu -\frac12 \sigma^{2})t_{k} + \sigma W(t_{k}) \}$,  we have
	\begin{displaymath}
	\lambda_{0} \ge \lim_{k\to\infty} \frac{1}{t_{k}} \E_{x}[c_{1} \eta(t_{k})] = \lim_{k\to\infty} \frac{1}{t_{k}} \E_{x}[c_{1} X_{0}(t_{k})]
	= \lim_{k\to\infty} \frac{1}{t_{k}} c_{1} x e^{\mu t_{k}} = \infty.
	\end{displaymath}

When $\mu =0$, in the presence of Assumption \ref{h-assumption} below,  we no longer have $\lim_{x\to\infty} (h(x) +c_{1} \mu x )  <0 $, thus violating Assumption \ref{assumption-pi_12} (i). Consequently  the vanishing discount method is not applicable.  
We will present in  Section \ref{sec-mu=0}   a case study  which indicates for the Cobb Douglas function $h(x) = x^{p}, p\in (0,1)$, the optimal long-term average reward can be arbitrarily large. 	
\end{rem}

We need the following assumption throughout this section:
\begin{assumption}\label{h-assumption}
	The function $h$ is nonnegative 
	 and satisfies the following conditions:
		\begin{itemize}
			\item[(i)] $h$ is strictly concave, continuously differentiable and nondecreasing on $[0,\infty)$, and satisfies $h(0)=0$; 
			\item[(ii)] $h$ has a finite Legendre transform on $(0,\infty )$; that is,   for all $z > 0$
			\begin{equation}\label{Legendre transform}
			\Phi_{h}(z) : = \sup_{x\geq 0}\left\{ h( x) - z x\right\}< \infty.
			\end{equation}
		\item[(iii)] $h$ satisfies the Inada condition at 0, i.e., \begin{equation}\label{e:h-Inada}\lim_{x\downarrow 0} \frac{h(x)}{x} = \infty.\end{equation}
		\end{itemize}
	\end{assumption}
	
\begin{rem}\label{rem-h-assumption}
Assumption \ref{h-assumption} is common in the finance literature; see, for example, \cite{GuoP-05} and \cite{DeAngelisF-14}. The Inada condition at 0 implies that the right derivative $h'(0+)$ of the  function $h$ is infinity. If this condition is violated, say, there   exists some $c\in [c_{1}, c_{2}]$ such that \begin{equation}\label{e:h'small}h'(0+) + \mu c \le 0,\end{equation} then one can show that  $\lambda_{0} =0$. 

Indeed, in the presence of  \eqref{e:h'small}, we can immediately verify  that the function $u(x) = c x$ and $\lambda =0$ satisfy \eqref{e:HJB-lta}. Consequently it follows from Theorem \ref{thm-verification-lta} that $\lambda_{0} =0$. Moreover, we can construct an optimal policy in the following way. Suppose $X(0-) = x \ge 0$. Let $\xi(t) \equiv 0$ and $\eta(t) \equiv \Delta\eta(0) = x$; that is, the policy harvests all and brings the state to 0 at time 0. The long-term average reward for such a policy is 0.

\end{rem}


	The main result of this section is presented next.

\begin{theorem}
Under Assumption \ref{h-assumption},
 there exist $0 < a_{*} < b_{*} < \infty$ such that the reflected geometric Brownian motion on the interval $[a_{*}, b_{*}]$ is an optimal state process $($if the initial point is outside this interval, then there will be an initial jump to the nearest point of the interval$)$. In other words, the optimal value $\lambda_{0}$ of \eqref{def:lambda_0-gbm} is achieved by  $\vphi_{*}(\cdot): = L_{a_{*}}(\cdot)- L_{b_{*}}(\cdot)$, in which $ L_{a_{*}}(\cdot)$ and $  L_{b_{*}}(\cdot)$ denote respectively the local times of $X$ at $a_{*}$ and $b_{*}$. Moreover, the Abelian and the Ces\`aro limits \eqref{e:abel-cesaro limits}  hold.
 \end{theorem}
 \begin{proof}
In view of Theorem \ref{thm-main result}, we only need to verify Assumptions \ref{assumption-pi_12}, \ref{assumption-V_r},  and  \ref{assumption-mu-sigma} hold.  Assumption   \ref{assumption-mu-sigma}  obviously holds. Under Assumption \ref{h-assumption}, it is established in \cite{GuoP-05} that there exist two positive numbers $a_{r} < b_{r}$ so that the discounted problem $V_{r}$  is a classical solution to \eqref{e2:HJB-r},  which verifies
Assumption \ref{assumption-V_r}.

 It remains to show Assumption \ref{assumption-pi_12} holds.  Since  $h(0) =0$,  we have $\lim_{x\downarrow 0} [h(x) + c_{2} \mu x] =0$. In view of \eqref{Legendre transform} and the fact that $\mu < 0$, we have $h(x) - ( \frac{-c_{1}\mu}{2} )x \le \Phi_{h}(\frac{-c_{1}\mu}{2} ) < \infty$. Consequently	it follows that \begin{displaymath}
\lim_{x\to\infty} [h(x) + c_{1}\mu x] \le \lim_{x\to\infty}  \bigg[\Phi_{h}\bigg(\frac{-c_{1}\mu}{2} \bigg) - \frac{c_{1}\mu}{2} x + c_{1}\mu x\bigg] =-\infty.
\end{displaymath}
Assumption  \ref{assumption-pi_12} (i) is therefore verified. To verify  Assumption  \ref{assumption-pi_12} (ii), we recall  that the scale and speed densities of the geometric Brownian motion are respectively given by $s(x)=x^{-\frac{2\mu}{\sigma^{2}}} $ and $m(x)= \frac{1}{\sigma^{2}} x^{\frac{2\mu}{\sigma^{2}}-2}$, $x> 0$. Next we  observe that condition \eqref{e:h-Inada} implies that there exists a positive constant $\delta$ so that $\frac{h(y)}{y} > -2\mu c_{2}$ for all $y\in (0, \delta)$. Now let $b > \delta > a$. Then  we have \begin{align*}
\frac{c_{1}}{2s(b)} -  \frac{c_{2}}{2s(a)} +  \int_{a}^{b}h(y) m(y) \d y   & \ge  -  \frac{c_{2}}{2s(a)} +  \int_{a}^{\delta}\frac{h(y)}{y} y \frac{1}{\sigma^{2}}y^{\frac{2\mu}{\sigma^{2}}-2} \d y    \\
  & >  -  \frac{c_{2}}{2s(a)} + \frac{1}{\sigma^{2}}  \int_{a}^{\delta}( -2\mu c_{2})  y^{\frac{2\mu}{\sigma^{2}}-1} \d y\\
  & =  -  \frac{c_{2}}{2s(a)} + c_{2} \bigg( \frac{1}{s(a)} -\frac{1}{s(\delta)}\bigg) \\
  & =  \frac{c_{2}}{2s(a)}- \frac{c_{2}}{s(\delta)} > 0,
\end{align*}
by choosing $a > 0$ sufficiently small so that $\frac{1}{s(a)} > \frac{2}{s(\delta)}$. This gives \eqref{e:lam>0condition}  and hence verifies Assumption  \ref{assumption-pi_12} (ii).
\end{proof}

\subsection{Case Study: $\mu=0$ and $h(x) = x^{p}, p\in (0, 1)$}\label{sec-mu=0}
In this subsection, we consider the problem \eqref{def:lambda_0-gbm} in which $\mu=0$ and $h(x) = x^{p}, p\in (0, 1)$. We will demonstrate via probabilistic arguments that $\lambda_{0} =\infty$. To this end, we first present the following lemma.

		\begin{lemma}\label{lem-mu=0}
		Suppose $\mu =0$ and $h(x) = x^{p}$ for some $p \in (0,1)$, then for all $\lambda > 0$ sufficiently large satisfying \begin{equation}
		\label{e:lambda-large}
		\frac{2p}{\sigma^{2}(1-p)} \lambda^{1-\frac1p} < \frac{c_{2}-c_{1}}{2} \quad \text{ and } \quad c_{2} \lambda^{\frac1p-1} > \frac{2}{\sigma^{2}p} \ln \lambda,
		\end{equation} the function \begin{equation}
		\label{e:u(x)defn}
		u(x) : =\begin{cases}\frac{2}{\sigma^{2} p (1-p)} x^{p} -\frac{2\lambda}{\sigma^{2}} \log x + C x, &\text{ if }x > a_{*},\\ c_{2}(x-a_{*}) + u(a_{*}), & \text{ if } 0\le x \le a_{*} \end{cases}
		\end{equation} is   twice continuously differentiable, nonnegative,  and satisfies the following system
		\begin{equation}\label{e:exm-qvi}
		\begin{cases}
		\LL u(x) + h(x)  - \lambda =0,      & \text{ if  } x > a_{*},\\
		c_{1} \le u'(x) \le c_{2},   & \text{ if }x > a_{*}, \\
		\LL u(x) + h(x)  - \lambda \le 0,      & \text{ if  } x \le a_{*},\\
		u'(x) = c_{2},   & \text{ if  } x \le a_{*}.
		\end{cases}
		\end{equation} Here $a_{*} = \lambda^{1/p}$ and $C= c_{2} -\frac{2p}{\sigma^{2} (1-p)} \lambda^{1-\frac1p}  $. \end{lemma}
	
	\begin{proof}
		Note that condition \eqref{e:lambda-large} implies that $C> c_{2} - \frac{c_{2}-c_{1}}{2} = \frac{c_{2}+c_{1}}{2} > c_{1}$. The fact that  $u \in C^{1}((0,\infty)) \cap C^{2}((0,\infty)\setminus\{a_{*}\})$ is a direct consequence of the definition of $u$ in \eqref{e:u(x)defn}. Detailed calculations reveal that $\lim_{x\downarrow a_{*}} u''(x) = 0$ and hence $u \in C^{2}(0,\infty)$. In addition, the equalities in \eqref{e:exm-qvi} can be verified by direct computations. We next show the inequalities in \eqref{e:exm-qvi} hold as well.
		
		For $x< a_{*}$, we have
		\begin{align*}
		\LL u(x) + h(x)  - \lambda  &  = h(x) -\lambda \le h(a_{*}) -\lambda =0.
		\end{align*} For $x> a_{*}$, we have \begin{align*}
		&u'(x) = \frac{2}{\sigma^{2}(1-p)} x^{p-1} - \frac{2\lambda}{\sigma^{2}} x^{-{1}} + C,\\
		& u''(x) = -\frac{2}{\sigma^{2} } x^{p-2}+ \frac{2\lambda}{\sigma^{2}} x^{-2 } = \frac{2\lambda}{\sigma^{2}} x^{-2 }  (1- x^{p}) \le 0.
		\end{align*} These  imply  that $u'(x) \le u'(a_{*}) = c_{2}$ for all $x> a_{*}$. Also we notice that $\lim_{x\to\infty} u'(x) = C > c_{1}$. Hence it follows that $c_{1} < u'(x) \le c_{2}$ for all  $x> a_{*}$.
		
		Also observe that $u(a_{*}) = \frac{2}{\sigma^{2} p}(\frac{\lambda}{1-p} -\lambda \ln \lambda) + C \lambda^{\frac1p}= \frac{2}{\sigma^{2} (1-p)}(\frac{1}{p} - p) \lambda + (c_{2}\lambda^{\frac1p-1}- \frac{2}{\sigma^{2}p} \ln \lambda ) \lambda > 0$ thanks to the second inequality of \eqref{e:lambda-large}. Hence   $u(x) > 0$ for all $x>0$.  The proof is complete.
	\end{proof}
	
	\begin{prop}
		Suppose $\mu =0$ and $h(x) = x^{p}$ for some $p \in (0,1)$, then $\lambda_{0} =\infty$.
	\end{prop}

	\begin{proof}	
	For any $\lambda > 0$ satisfying  \eqref{e:lambda-large}, we
 construct a policy whose long-term average reward  equals $\lambda$. Since $\lambda $ is arbitrary, it follows that $\lambda_{0} =\infty$. To this end, recall that $a_{*} = \lambda^{1/p}$.  Let $b_{*}: = \log a_{*}$ and consider the drifted Brownian motion reflected at $b_{*}$:
\begin{displaymath}
\psi(t) : = b_{*} -\frac12 \sigma^{2} t + \sigma W(t)  + \phi(t) \ge b_{*},  \quad t\ge 0
\end{displaymath}		where $\phi(t)  
=  - \inf_{0\le s \le t}\{ - \frac12 \sigma^{2} s + \sigma W(s)\}$. One can verify immediately that $\phi$ is a nondecreasing process that increases only when $\psi(t) =b_{*}$; i.e., \begin{equation}\label{e:phi-sde}
\phi(t) =  \int_{0}^{t} \1_{\{ b_{*}\}} (\psi(s)) \d \phi(s).
\end{equation}

Now let $X(t) : = e^{\psi(t)}$.	Note that $X(0) = a_{*}$. By It\^o's formula, we have
\begin{equation}
\label{e:X(t)-sde}
\d X(t) =\bigg ( -\frac{\sigma^{2}}2 + \frac{\sigma^{2}}2  \bigg )X(t)  \d t +  \sigma X(t) \d W(t) + X(t) \d  \phi(t)=  \sigma X(t) \d W(t) +   \d  L_{a^*}(t),
\end{equation} where   $ L_{a_{*}}(t): = \int_{0}^{t} X(s) \d \phi(s)$.   We
make the following observations.
\begin{enumerate}
  \item[(i)] $X(t) \ge a_{*}$ for all $t\ge 0$. This is obvious since $\psi(t) \ge b_{*}$ for all $t\ge 0$.
  \item[(ii)]   $ L_{a_{*}}$ is a nondecreasing process that increases only when $X(t) =a_{*}$. Indeed, using \eqref{e:phi-sde} and note that $X(t) =a_{*}$ if and only if $\psi(t) =b_{*}$, it follows that
 \begin{equation}
\label{e:Local-time a}
 \begin{aligned}
 L_{a_{*}}(t)& =  \int_{0}^{t} X(s) \d \phi(s)=  \int_{0}^{t} X(s)  \1_{\{ b_{*}\}} (\psi(s)) \d  \phi(s)
\\& =  \int_{0}^{t}   \1_{\{ a_{*}\}} (X(s)) X(s) \d  \phi(s)= \int_{0}^{t}   \1_{\{ a_{*}\}} (X(s)) \d  L_{a^*}(s).
\end{aligned}
\end{equation} \end{enumerate}

 We now apply It\^o's formula to the process $u(X(t))$, where $X$ is given by \eqref{e:X(t)-sde} with $X(0) = a_{*}$.  Since $X(t) \ge a_{*}$ for all $t\ge 0$ and $u$ satisfies   \eqref{e:exm-qvi}, we have \begin{align*}
\E[u(X(t\wedge \beta_{n})]  & = u(a_{*}) + \E\bigg[\int_{0}^{t\wedge \beta_{n}} \LL u(X(s)) \d s  + \int_{0}^{t\wedge \beta_{n}} u'(X(s)) \d L_{a_{*}}(s) \bigg]   \\
  &  =  u(a_{*}) + \E\bigg[-\int_{0}^{t\wedge \beta_{n}} h(X(s)) \d s + c_{2} L_{a_{*}}(t) \bigg] + \lambda \E[t\wedge \beta_{n}],
\end{align*} where $\beta_{n}: = \inf\{ t\ge 0: X(t) \ge n\}$ and $n\in \N \cap (a_{*},\infty)$. Since   $u $ is nonnegative, rearranging the terms yields \begin{displaymath}
\E\bigg[ \int_{0}^{t\wedge \beta_{n}} h(X(s)) \d s - c_{2} L_{a_{*}}(t) \bigg] \le  u(a_{*})+   \lambda \E[t\wedge \beta_{n}].
\end{displaymath}
Now dividing both sides by $t$,  and passing to the limit first as $n\to \infty$  and then  as $t\to\infty$, we obtain
\begin{align}\label{e:<=lambda}
\limsup_{t\to\infty} \frac1t  \E\bigg[\int_{0}^{t} h(X(s)) \d s - c_{2} L_{a_{*}}(t)   \bigg]   \le \lambda.
\end{align}

 In view of Section 15.5 of \cite{KarlinT81}, the process $X$ of \eqref{e:X(t)-sde} has a unique  stationary distribution $\pi(\d x) =a_{*} x^{-2} \d x, x\in (a_{*},\infty) $. The strong law of large numbers (\cite[Theorem 4.2]{khasminskii-12}) then implies that as $t\to \infty$\begin{displaymath}
\frac1t\int_{0}^{t}h_{n}(X(s))\d s \to \int_{a_{*}}^{\infty} h_{n}(x) \pi (\d x), 
\  \text{ a.s.}
\end{displaymath} where $h_{n}(x) : = h(x) \wedge n$ and $n\in \N$. For each $n\in \N$, the random variables $\{\frac1t \int_{0}^{t} h_{n}(X(s)) \d s, t> 0 \}$ are non-negative and bounded above by $n$. Thus  we can apply the bounded convergence theorem to obtain
\begin{displaymath}
\lim_{t\to\infty} \E\bigg[ \frac1t \int_{0}^{t} h_{n}(X(s)) \d s\bigg] =  \int_{a_{*}}^{\infty} h_{n}(x) \pi (\d x).
\end{displaymath} On the other hand, since $h\ge h_{n}$ and $h_{n} \uparrow h$, we have
\begin{align}\label{e:running reward average}
\nonumber \liminf_{t\to\infty} \E\bigg[ \frac1t \int_{0}^{t} h(X(s)) \d s\bigg] &= \lim_{n\to\infty} \liminf_{t\to\infty} \E\bigg[ \frac1t \int_{0}^{t} h(X(s)) \d s\bigg]\\ \nonumber&  \ge \lim_{n\to\infty} \liminf_{t\to\infty} \E\bigg[ \frac1t \int_{0}^{t} h_{n}(X(s)) \d s\bigg] \\
\nonumber&  = \lim_{n\to\infty}  \int_{a_{*}}^{\infty} h_{n}(x) \pi (\d x)\\
&  \nonumber =  \int_{a_{*}}^{\infty} h(x) \pi (\d x)\\ &= \frac{(a_{*})^{p}}{1-p} = \frac{\lambda}{1-p} \ge \lambda,
\end{align} where the second to the last line 
follows from the monotone convergence theorem and the equality in the last line follows from the fact that  $a_{*} = \lambda^{1/p}$.

We next estimate $\E[L_{a_{*}}(t)]$. To this end,   we    observe from the third equality of \eqref{e:Local-time a} that $$L_{a_{*}}(t)  =  \int_{0}^{t}   \1_{\{ a_{*}\}} (X(s)) X(s) \d  \phi(s)= a_{*} \phi(t),$$ and hence  $\E[L_{a_{*}}(t)]= a_{*} \E[\phi(t)] $. On the other hand, since   $\phi(t)
=  - \inf_{0\le s \le t}\{ - \frac12 \sigma^{2} s + \sigma W(s)\}= \sup_{0\le s \le t}\{  \frac12 \sigma^{2} s + \sigma (-W(s))\}$,  it follows from (1.8.11) of \cite{Harrison-85} that \begin{displaymath}
\P\{\phi(t) \le y \} = \Phi\bigg(\frac{y-\frac12 \sigma^{2}t}{\sigma \sqrt t}\bigg) - e^{y} \Phi\bigg(\frac{-y-\frac12 \sigma^{2}t}{\sigma \sqrt t}\bigg),
\end{displaymath}  where $\Phi$ is the standard normal distribution function $\Phi(x): = \int_{-\infty}^{x} \frac1{\sqrt{2\pi}} e^{-\frac{z^{2}}{2}}\d z$. Then 
we can compute \begin{align}
\label{e:phi(t)-mean}
 \nonumber\E[\phi(t)]   & = \int_{0}^{\infty} y\bigg( \frac{1}{\sqrt{2\pi\sigma^{2}t}} e^{-\frac{(y-t\sg^{2}/2)^{2}}{2\sg^{2} t}} -e^{y} \Phi\bigg( \frac{-y-t\sg^{2}/2}{\sg\sqrt t}\bigg) + e^{y}  \frac{1}{\sqrt{2\pi\sigma^{2}t}} e^{-\frac{(y+t\sg^{2}/2)^{2}}{2\sg^{2} t}}  \bigg)\d y   \\
  & = \sg^{2} t \Phi(\sg\sqrt t/2) +\frac{\sqrt {2t}\sg e^{-\frac{\sg^{2}t}{8}}}{\sqrt{\pi }}- \int_{0}^{\infty} y e^{y} \Phi\bigg( \frac{-y-t\sg^{2}/2}{\sg\sqrt t}\bigg) \d y.
\end{align}
To estimate the last integral, we use the tail estimate for the standard normal distribution function (see, for example, \cite[Section 7.1]{FellerV1-68}):
\begin{displaymath}
1- \Phi(x) = \int_{x}^{\infty} \frac{1}{\sqrt{2\pi}} e^{-\frac{y^{2}}{2}} \d y \ge [x^{-1} - x^{-3}] \frac{1}{\sqrt{2\pi}} e^{-\frac{x^{2}}{2}}, \quad x> 0,
\end{displaymath} to obtain \begin{displaymath}
\Phi\bigg( \frac{-y-t\sg^{2}/2}{\sg\sqrt t}\bigg) = 1- \Phi\bigg( \frac{y+t\sg^{2}/2}{\sg\sqrt t}\bigg) \ge   \frac{\sg\sqrt t}{y+t\sg^{2}/2}\bigg[1- \frac{\sg^{2}  t}{(y+t\sg^{2}/2)^{2}} \bigg]\frac{1}{\sqrt{2\pi}} e^{-\frac{(y+t\sg^{2}/2)^{2}}{2\sg^{2} t}}.
\end{displaymath}
Moreover, we have
\begin{displaymath}
\lim_{y\to\infty} y  \frac{\sg\sqrt t}{y+t\sg^{2}/2}\bigg[1- \frac{\sg^{2}  t}{(y+t\sg^{2}/2)^{2}} \bigg]= \sigma\sqrt t.
\end{displaymath} Therefore, for any $\e > 0$, there exists an $M>0$ so that \begin{displaymath}
 y  \frac{\sg\sqrt t}{y+t\sg^{2}/2}\bigg[1- \frac{\sg^{2}  t}{(y+t\sg^{2}/2)^{2}} \bigg] \ge  \sigma\sqrt t (1-\e)\quad  \text{ for all } y\ge M.
\end{displaymath} Then we can compute
\begin{align*}
\int_{0}^{\infty} y e^{y} \Phi\bigg( \frac{-y-t\sg^{2}/2}{\sg\sqrt t}\bigg) \d y  &  \ge \int_{M}^{\infty} y e^{y} \Phi\bigg( \frac{-y-t\sg^{2}/2}{\sg\sqrt t}\bigg) \d y \\ & \ge  \int_{M}^{\infty} y e^{y}   \frac{\sg\sqrt t}{y+t\sg^{2}/2}\bigg[1- \frac{\sg^{2}  t}{(y+t\sg^{2}/2)^{2}} \bigg]\frac{1}{\sqrt{2\pi}} e^{-\frac{(y+t\sg^{2}/2)^{2}}{2\sg^{2} t}}\d y \\
  & \ge \int_{0}^{\infty}\sg\sqrt t (1-\e) \frac{1}{\sqrt{2\pi}} e^{y }e^{-\frac{(y+t\sg^{2}/2)^{2}}{2\sg^{2} t}}\d y \\ & \ \ - \int_{0}^{M}\sg\sqrt t (1-\e) \frac{1}{\sqrt{2\pi}} e^{y }e^{-\frac{(y+t\sg^{2}/2)^{2}}{2\sg^{2} t}}\d y\\
  & \ge \sg\sqrt t (1-\e)\sg \sqrt t \Phi(\sg\sqrt t/2)- \sg\sqrt t (1-\e) \frac{M e^{M}}{\sqrt{2\pi}}.
\end{align*} Plugging this estimation into \eqref{e:phi(t)-mean} gives \begin{align*}
\E[\phi(t)] &\le \sg^{2} t \Phi(\sg\sqrt t/2) +\frac{\sqrt {2t}\sg e^{-\frac{\sg^{2}t}{8}}}{\sqrt{\pi }}-  \sg\sqrt t (1-\e)\sg \sqrt t \Phi(\sg\sqrt t/2)+  \sg\sqrt t (1-\e) \frac{M e^{M}}{\sqrt{2\pi}} \\
& =  \sg^{2} t \Phi(\sg\sqrt t/2) \e  +\frac{\sqrt {2t}\sg e^{-\frac{\sg^{2}t}{8}}}{\sqrt{\pi }}+  \sg\sqrt t (1-\e) \frac{M e^{M}}{\sqrt{2\pi}}.
\end{align*}
Hence, we have \begin{equation}
\label{e:La(t) average}
\limsup_{t\to \infty} \frac1t\E[L_{a_{*}}(t)] = \limsup_{t\to \infty} \frac1t\E[ a_{*} \phi(t)] \le a_{*}\sg^{2}\e.
\end{equation}

Now combining \eqref{e:running reward average} and \eqref{e:La(t) average} yields
\begin{displaymath}
\liminf_{t\to\infty}  \E\bigg[ \frac1t \int_{0}^{t} h(X(s)) \d s -c_{2}L_{a_{*}}(t) \bigg] \ge   \lambda- c_{2} a_{*}\sg^{2}\e.
\end{displaymath} Since $\e> 0$ is arbitrary, it follows that \begin{displaymath}
\liminf_{t\to\infty}  \E\bigg[ \frac1t \int_{0}^{t} h(X(s)) \d s -c_{2}L_{a_{*}}(t) \bigg] \ge   \lambda.
\end{displaymath} This, together with \eqref{e:<=lambda}, implies that the policy $L_{a_{*}}$ of \eqref{e:Local-time a} has a long-term average reward of $\lambda$.
	The proof is concluded.
\end{proof}

\section{A Direct Solution Approach}\label{sect-direct-soln}
In the previous sections, we obtained the solution to the ergodic two-sided singular control problem \eqref{def: lambda_0} via the vanishing discount method. A critical condition for this approach is that we need to first solve the discounted problem \eqref{e-V_r(x) defn}. In this section, we provide a direct solution approach to the problem \eqref{def: lambda_0}.

Let us briefly explain the strategy on how to derive $\lambda_{0}$ of \eqref{def: lambda_0}. First we focus on  a class of  policies that keeps the process in the interval $[a, b] $. The proof of Corollary \ref{cor-lambda_0>0}  shows that the long-term average  reward for such a policy  is equal to $\lambda(a, b)$ of \eqref{e1:lambda}. Next we impose conditions so that $\lambda(a, b)$ achieves its   maximum value  $\lambda_{*} = \lambda(a_{*}, b_{*})$ at a pair $0< a_{*} < b_{*}< \infty$. The maximizing pair $(a_{*}, b_{*})$  further allows us to derive a $C^{2}$ solution to the HJB equation \eqref{e:HJB-lta}.  This, together with the  verification  theorem (Theorem \ref{thm-verification-lta}), reveals that $\lambda_{0} = \lambda_{*}$ and the $(a_{*}, b_{*})$-reflection policy is an optimal policy.

This approach is motivated by   the recent paper \cite{Alva-18}, which solves an ergodic  two-sided singular control problem for  general one-dimensional diffusion processes. Note that the setups in  \cite{Alva-18} is different from ours. In particular, the cost rates  associated with the  singular controls are all positive in \cite{Alva-18}. By contrast,   our formulation in \eqref{def: lambda_0}  has mixed signs for the singular controls $\eta$ and $\xi$. We also note that  Theorem 2.5 in \cite{Alva-18} only proves that the $(a_{*}, b_{*})$-reflection policy is optimal {\em in the class of two-point reflection policies.} Here we observe that   the  $(a_{*}, b_{*})$-reflection policy is optimal among all admissible controls by the verification theorem.

Recall from the proof of Corollary \ref{cor-lambda_0>0}   that   the long-term average  reward for the $(a, b)$-reflection  policy  is equal to $\lambda(a, b)$ of \eqref{e1:lambda}.  Now we wish to maximize the long-term average reward
  $\lambda(a, b) $.  First, detailed computations reveal that
\begin{align*}
 \frac{\partial}{\partial a} \lambda(a, b)& = \frac{m(a)}{M[a, b]} [ \lambda(a, b)- \pi_{2}(a)], \\
  \frac{\partial}{\partial b} \lambda(a, b)& = \frac{m(b)}{M[a, b]} [\pi_{1}(b) -\lambda(a, b) ].
\end{align*} where \begin{equation}\label{e:pi-12-defn}
\pi_{1}(x) = h(x) + c_{1} \mu(x), \quad \text{ and }\quad  \pi_{2}(x) = h(x) + c_{2} \mu(x).
\end{equation}
 Therefore the first order optimality condition says that a maximizing pair $a_{*} < b_{*}$ must satisfy
 \begin{equation}
\label{e-1st-order}
\lambda(a_{*}, b_{*} ) = \pi_{1}(b_{*}) = \pi_{2} (a_{*}).
\end{equation}
Furthermore, \eqref{e-1st-order} is equivalent to
\begin{align}
\label{e1-1st-order}
 &   \int_{a_{*}}^{b_{*}} (\pi_{2}(x) - \pi_{2}(a_{*})) m(x) \d x + \frac{c_{1} - c_{2}}{2 s(b_{*})} =0,     \\
 \label{e2-1st-order}   &   \int_{a_{*}}^{b_{*}} (\pi_{1}(x) - \pi_{1}(b_{*})) m(x) \d x + \frac{c_{1} - c_{2}}{2 s(a_{*})}=0.
\end{align} To see this, we note that using the definition of  $\lambda(a, b)$ in \eqref{e1:lambda}, we can write
\begin{align*}
  \frac{\partial}{\partial b} \lambda(a, b)& = \frac{m(b)}{M[a, b]} [\pi_{1}(b) -\lambda(a, b) ] \\
& = \frac{m(b)}{M^{2}[a, b]} \bigg(\int_{a}^{b} \pi_{1}(b)  m(x) \d x - \int_{a}^{b} h(x) m(x) \d x - \frac{c_{1}}{2 s(b) } + \frac{c_{2}}{ 2 s(a)} \bigg)\\
& = \frac{m(b)}{M^{2}[a, b]} \bigg(\int_{a}^{b}( \pi_{1}(b) - \pi_{1}(x) )  m(x) \d x  + \int_{a}^{b} c_{1} \mu(x)  m(x) \d x  - \frac{c_{1}}{2 s(b) } + \frac{c_{2}}{ 2 s(a)} \bigg)\\
& = \frac{m(b)}{M^{2}[a, b]} \bigg(\int_{a}^{b}( \pi_{1}(b) - \pi_{1}(x) )  m(x) \d x + \frac{c_{2} - c_{1}}{ 2 s(a) } \bigg).
\end{align*} This gives the equivalence between $\lambda(a_{*}, b_{*} ) = \pi_{1}(b_{*}) $ and \eqref{e2-1st-order}. The equivalence between $\lambda(a_{*}, b_{*} ) = \pi_{2}(a_{*}) $ and \eqref{e1-1st-order} can be established in a similar fashion.

To show that the first order condition \eqref{e-1st-order} or equivalently the system of equations  \eqref{e1-1st-order}--\eqref{e2-1st-order} has a solution, we impose the following conditions:

\begin{assumption}\label{assumpt-pi12-sect7}
\begin{itemize}
  \item[(i)]  The functions $h$ and $\mu$ are  continuously differentiable  with $h'(x) > 0$ for all $x >0$.  In addition,  $h(0) = \mu(0) =0$.
  \item[(ii)] For $i =1,2$, there exists an $\hat x_{i} > 0$ such that $\pi_{i} (x)$ is strictly  increasing on $(0, \hat x_{i})$ and strictly decreasing on $[\hat x_{i}, \infty) $. In addition $\lim_{x\to \infty} \pi_{1}(x) < 0$. Consequently, there exists a $b_{0} > \hat x_{1} $ such that $\pi_{1} (b_{0}) =0$.
  \item[(iii)] It holds true that
  \begin{align}
\label{e:key-condition}
  \liminf_{a\downarrow 0} \bigg[ \int_{0}^{b_{0}}[\pi_{1}(y) - \pi_{1}(b_{0})]  m(y)\d y +  \frac{c_{1} - c_{2}}{2s(a)}\bigg] > 0.
\end{align}
\end{itemize}
\end{assumption}

\begin{rem}\label{rem1-about-pi-12}
We remark that  Assumption \ref{assumpt-pi12-sect7} (i)  is  standard in the literature of singular controls; see, for example, \cite{Alva-18, AlvaH-19} and \cite{Weerasinghe-07}. Condition (ii) is motivated by   similar conditions in  \cite{Alva-18} and \cite{Weerasinghe-07}. In addition, these conditions are satisfied under the  setup in \cite{GuoP-05} as well.   Condition (iii) 
 is a technical one and it guarantees the existence of an optimizing pair $0 < a_{*} < b_{*} < \infty$ satisfying   the system of equations \eqref{e1-1st-order}--\eqref{e2-1st-order}; see the proof of Proposition \ref{prop-a*b*}. 
\end{rem}

\begin{rem}\label{rem2-about-pi-12}
We also note that  the extreme points 
$\hat x_{2},   \hat x_{1}$ in Assumption \ref{assumpt-pi12-sect7} must satisfy $\hat x_{2} <  \hat x_{1}$. Suppose on the contrary that $\hat x_{1} \le  \hat x_{2}$. Then since $\pi_{1}$ achieves its maximum at $\hat x_{1}$ and $\pi_{2} $ is strictly increasing on $(0, \hat x_{2})$, we have \begin{equation}\label{e:pi1pi2-derivatives}
 \pi_{1}' (\hat x_{1} ) = h'(\hat x_{1} ) + c_{1} \mu'(\hat x_{1})  =0,\quad  \text{ and } \quad   \pi_{2}' (\hat x_{1} ) = h'(\hat x_{1} ) + c_{2} \mu'(\hat x_{1})   \ge  0.
\end{equation} On the other hand, Assumption \ref{h-assumption} (i) says that $h'(\hat x_{1}) >  0$. This, together with fact that $0 < c_{1} < c_{2} < \infty$, implies that
\begin{displaymath}
 \pi_{2}' (\hat x_{1} ) = h'(\hat x_{1} ) + c_{2} \mu'(\hat x_{1}) =  h'(\hat x_{1} ) - c_{2} \frac{h'(\hat x_{1} ) }{c_{1}} =  - \frac{c_{2} - c_{1}}{ c_{1}}  h'(\hat x_{1} ) <  0;
\end{displaymath} resulting in a contradiction to \eqref{e:pi1pi2-derivatives}.
\end{rem}

\begin{prop}\label{prop-a*b*}
Let Assumption \ref{assumpt-pi12-sect7} hold. Then there exists a unique 
pair $0 < a_{*} < b_{*} < \infty$ satisfying   the system of equations \eqref{e1-1st-order}--\eqref{e2-1st-order}.
\end{prop}

\begin{proof} Since $\hat x_{2} <  \hat x_{1}$ thanks to  Remark \ref{rem2-about-pi-12}, we  only need to consider two cases: $\pi_{2} (\hat x_{2}) \ge \pi_{1} (\hat x_{1})$ and $\pi_{2} (\hat x_{2}) <  \pi_{1} (\hat x_{1})$.

{\em Case (i)}: $\pi_{2} (\hat x_{2}) \ge \pi_{1} (\hat x_{1})$. In this case, thanks to Assumption \ref{assumpt-pi12-sect7}, there exists a $y_{1} \in (0, \hat x_{2}] $ such that $\pi_{2}(y_{1} ) = \pi_{1} (\hat x_{1}) $. 
In view of Remark \ref{rem2-about-pi-12}, we have $y_{1} \le \hx_{2} < \hx_{1}$.  In addition, Assumption \ref{assumpt-pi12-sect7} also implies that for any $a\in [0, y_{1}]$, there exists a unique $b_{a} \in [\hx_{1}, b_{0})$ such that $\pi_{1}(b_{a}) = \pi_{2}(a)$. Consequently we can   consider the function
\begin{equation}
\label{e:l(a)fn-defn}
\ell(a): = \int_{a}^{b_{a}} h(x) m(x) \d x - \pi_{1} (b_{a}) M[a, b_{a}] + \frac{c_{1}}{ 2 s(b_{a})} - \frac{c_{2}}{2 s(a)}, \quad a\in [0, y_{1}].
\end{equation} Using \eqref{e:convenient_identity} and the fact that  $\pi_{1}(b_{a}) = \pi_{2}(a)$, we can rewrite $\ell(a)$ as \begin{displaymath}
\ell(a) = \int_{a}^{b_{a}} [\pi_{2}(x) - \pi_{2}(a)] m(x) \d x + \frac{c_{1}-c_{2}}{2s(b_{a})} =  \int_{a}^{b_{a}} [\pi_{1}(x) - \pi_{1}(b_{a})] m(x) \d x + \frac{c_{1}-c_{2}}{2s({a})}.
\end{displaymath}
 Since $b_{y_{1}}= \hx_{1}$, we have
\begin{align*}
\ell(y_{1})  
& =  \int_{y_{1}}^{\hx_{1}} [\pi_{1} (x) - \pi_{1}(\hx_{1})]m(x)\d x + \frac{c_{1} - c_{2}}{2 s(y_{1})}.
\end{align*} Recall  from  Assumption \ref{assumpt-pi12-sect7} that $\pi_{1}$ is strictly increasing on $(0, \hx_{1})$ and $c_{1} < c_{2}$. Hence it follows that $\ell(y_{1}) < 0. $

Next we show that $\ell$ is decreasing on $[0, y_{1}]$ and hence $\ell(0+): = \lim_{a\downarrow 0} \ell(a)$ exists. To this end, we consider $0\le a_{1} < a_{2} \le y_{1}$ and denote $b_{i} = b_{a_{i}}$ for $i= 1, 2$. Since $\pi_{1} (b_{1}) = \pi_{2} (a_{1}) < \pi_{2}(a_{2}) = \pi_{1}(b_{2}) $ and $\pi_{1}$ is decreasing on $[\hx_{1}, \infty)$, it follows that $\hx_{1} \le b_{2} < b_{1}$. As a result, we can compute
\begin{align*}
 \ell(a_{1}) -\ell(a_{2})& = \int_{a_{1}}^{a_{2}} h(x) m(x)\d x +  \int_{b_{2}}^{b_{1}} h(x) m(x)\d x  - \pi_{1}(b_{1}) M[a_{1}, b_{1}] +  \pi_{1}(b_{2}) M[a_{2}, b_{2}] \\ & \qquad + \frac{c_{1}}{2s(b_{1})} - \frac{c_{2}}{2s(a_{1})}- \frac{c_{1}}{2s(b_{2})} + \frac{c_{2}}{2s(a_{2})}   \\
& =  \int_{a_{1}}^{a_{2}} [\pi_{2}(x) - \pi_{2}(a_{1})] m(x) \d x +  \int_{b_{2}}^{b_{1}} [\pi_{1}(x)- \pi_{1}(b_{1})] m(x)\d x\\ & \qquad + [\pi_{1}(b_{2}) - \pi_{1}(b_{1}) ] M[a_{2}, b_{2}].
\end{align*} Thanks to Assumption \ref{assumpt-pi12-sect7}, the terms  $\pi_{2}(x) - \pi_{2}(a_{1})$, $\pi_{1}(x)- \pi_{1}(b_{1})$, and $\pi_{1}(b_{2}) - \pi_{1}(b_{1})$ are all positive. Therefore $ \ell(a_{1}) -\ell(a_{2}) > 0$ as desired.

Using Assumption \ref{assumpt-pi12-sect7} (iii), we have  $\ell(0+) > 0$.  Since the function  $\ell$ is also continuous, there exists a unique $a_{*}\in (0, y_{1}]$ such that $\ell(a_{*}) =0$. Denote  $b_{*} = b_{a_{*}} \in [\hx_{1}, b_{0}) $. Then
\begin{align*}
0 & = \ell(a_{*}) = \int_{a_{*}}^{b_{*}} h(x) m(x) \d x - \pi_{1} (b_{*}) M[a_{*}, b_{*}] + \frac{c_{1}}{ 2 s(b_{*})} - \frac{c_{2}}{2 s(a_{*})} \\
 & =   \int_{a_{*}}^{b_{*}} [\pi_{2} (x) - \pi_{2} (a_{*}) ] m(x) \d x - \frac{c_{2}}{2s(b_{*})} +  \frac{c_{2}}{2s(a_{*})}  + \frac{c_{1}}{ 2 s(b_{*})} - \frac{c_{2}}{2 s(a_{*})} \\
&  =  \int_{a_{*}}^{b_{*}} [\pi_{2} (x) - \pi_{2} (a_{*}) ] m(x) \d x  + \frac{c_{1} - c_{2}}{2 s(b_{*})}.
\end{align*} This gives \eqref{e1-1st-order}. Similar calculations reveal that
\begin{align*}
0 & = \ell(a_{*}) 
=  \int_{a_{*}}^{b_{*}} [\pi_{1} (x) - \pi_{1} (b_{*}) ] m(x) \d x  + \frac{c_{1} - c_{2}}{2 s(a_{*})},
\end{align*} establishing \eqref{e2-1st-order}.

{\em Case (ii)}: $\pi_{2} (\hat x_{2}) <  \pi_{1} (\hat x_{1})$.  In this case, for any $a\in [0, \hx_{2}]$, there exists a $b_{a} \in [\hx_{1}, \infty)$ such that $\pi_{2}(a) = \pi_{1}(b_{a})$. Consequently we can define the function $\ell(a)$ as in \eqref{e:l(a)fn-defn} for all  $a\in [0, \hx_{2}]$. Let $y_{2}  \in [\hx_{1}, \infty)$ be such that $\pi_{2}(\hx_{2}) = \pi_{1}(y_{2})$. Then we have $b_{\hx_{2}} = y_{2}$ and
\begin{align*}
\ell(\hx_{2}) 
  & =  \int_{\hx_{2}}^{y_{2}} [ \pi_{2}(x) - \pi_{2}(\hx_{2})] m(x) \d x +  \frac{c_{1} - c_{2}}{2 s(y_{2})}  < 0.
\end{align*} The rest of the proof is very similar to that of Case (i). We shall omit the details for brevity.
\end{proof}

\begin{prop}\label{prop-HJB-soln}
Let Assumption  \ref{assumpt-pi12-sect7} hold. Then there exist a function $u \in C^{2}([0, \infty))$ and a positive number $\lambda_{*}$ satisfying the HJB equation \eqref{e:HJB-lta}.
\end{prop}
\begin{proof} Let $0 < a_{*} < b_{*} < \infty$ be as in Proposition \ref{prop-a*b*} and define
\begin{equation}
\label{eq:lambda*defn}
\lambda_{*}: = \frac{1}{M[a_{*}, b_{*}]} \bigg[\int_{a_{*}}^{b_{*}} h(x) m(x) \d x + \frac{c_{1}}{2s(b_{*})} -  \frac{c_{2}}{2s(a_{*})}  \bigg].
\end{equation}In addition, we consider the function
\begin{equation}
\label{e:soln-HJB}
u(x) =  \begin{cases}
   c_{1} x+ \int_{a_{*}}^{x} 2s(u) \int_{u}^{b_{*}}[\pi_{1}  (y) - \lambda_{*}] m(y) \d y\d u, & x\in [a_{*}, b_{*}],  \\
    c_{1} (x-b_{*}) + u(b_{*}),  &  x > b_{*}, \\
    c_{2} (x-a_{*}) + u(a_{*}), & x < a_{*}.
\end{cases}
\end{equation} We next  verify that the pair $(u, \lambda_{*})$ satisfies \eqref{e:HJB-lta}. First it is obvious that $u$ is continuously differentiable and satisfies $u'(x) = c_{1}$ for $x \ge  b_{*}$ and $u'(x) = c_{2}$ for $x \le a_{*}$.

Next we show that $u$ is $C^{2}$ and satisfies $\LL u( x) + h(x)-\lambda_{*}=0 $ for $x\in (a_{*}, b_{*})$. To this end, we compute for $x\in (a_{*}, b_{*})$:
\begin{align*}
 u'(x)   & = c_{1}+ 2 s(x)  \int_{x}^{b_{*}}[ \pi_{1}  (y)  - \lambda_{*}] m(y) \d y,   \\
 u''(x)    &  = -4s(x) \frac{\mu(x)}{\sigma^{2}(x)}  \int_{x}^{b_{*}}[ \pi_{1}  (y) - \lambda_{*}] m(y) \d y  + 2 s(x) [- \pi_{1} (x) + \lambda_{*}]m(x).
\end{align*} Consequently it follows that $u$ satisfies $ \LL u( x) + h(x)-\lambda_{*} =0$ for $x\in (a_{*}, b_{*}) $.

To show that $u$ is $C^{2}$, it suffices to show that $u$ is $C^{2}$ at the points   $a_{*}$ and $  b_{*}$.
Recall that $a_{*} < b_{*}$ and $\lambda_{*} $ satisfy \eqref{e-1st-order} thanks to Proposition \ref{prop-a*b*}. Therefore it follows immediately that $u'(b_{*}-) = c_{1} $ and  $u''(b_{*}-) = 0 $. In addition,      we can use \eqref{e2-1st-order} to compute  
\begin{align*}
u'(a_{*}+) =c_{1}+  2 s(a_{*})  \int_{a_{*}}^{b_{*}}[ \pi_{1}  (y)  - \lambda_{*}] m(y) \d y
  = c_{1} +  2 s(a_{*})  \frac{c_{2}- c_{1}}{2s(a_{*})}
 = c_{2},
\end{align*} and  
\begin{align*}
 u''(a_{*}+) &=  -4s(a_{*}) \frac{\mu(a_{*})}{\sigma^{2}(a_{*})}  \int_{a_{*}}^{b_{*}}[ \pi_{1}  (y)  - \lambda_{*}] m(y) \d y  + 2 s(a_{*}) [- \pi_{1}  (a_{*}) + \lambda_{*}]m(a_{*}) \\
& =  -4s(a_{*}) \frac{\mu(a_{*})}{\sigma^{2}(a_{*})}  \cdot \frac{c_{2}- c_{1}}{2s(a_{*})} + \frac{2}{\sigma^{2}(a_{*})} [- \pi_{1}  (a_{*}) + \lambda_{*}]\\
& = -\frac{2}{\sigma^{2}(a_{*})}[(c_{2}- c_{1} ) \mu(a_{*}) + \pi_{1} (a_{*}) - \lambda_{*}]\\
& =0,
\end{align*} where the last equality follows from  \eqref{e-1st-order}. Therefore we have shown that $u\in C^{2}((0,\infty))$.

Recall that the function $\pi_{1}$ is decreasing on $[\hx_{1}, \infty)$ and $b_{*} \in [\hx_{1}, \infty)$ by the proof of Proposition \ref{prop-a*b*}. Therefore for any $x\ge b_{*}$, \begin{align*}
  \LL u( x) + h(x)-\lambda_{*} = c_{1} \mu(x) + h(x) - \lambda_{*} = \pi_{1}(x) -\lambda_{*}  \le \pi_{1}(b_{*}) -\lambda_{*}  =0.
\end{align*} Similar argument reveals that $  \LL u( x) + h(x)-\lambda_{*} \le 0$ for all $x\in a_{*}$.

It remains to show that $u'(x) \in [c_{1}, c_{2}]$ for $x\in  (a_{*}, b_{*})$. To this end, we first consider the function $$k(x): =  \int_{x}^{b_{*}}[ \pi_{1}  (y)  - \lambda_{*}] m(y) \d y, x\in [a_{*}, b_{*}]. $$  We claim that $k$ is nonnegative, which, in turn, implies that $v'(x) = c_{1} + 2 s(x) k(x) \ge c_{1}$ on $[a_{*}, b_{*}]$.  To see the claim, we consider two cases. (i) If $\pi_{1} (y) \ge \lambda_{*} $ for all $y\in [a_{*}, b_{*}] $, then $k(x) \ge 0$ for all $x\in [a_{*}, b_{*}] $. (ii) Otherwise, since   $a_{*} \le \hx_{2} < \hx_{1}$, $b_{*} > \hx_{1}$, the monotonicity of $\pi_{1}$ implies that there exists a $y_{1} \in (a_{*}, b_{*})$ so that $ \pi_{1}  (y)  - \lambda_{*} $ is negative on $[a_{*}, y_{1})$ and positive on $(y_{1}, b_{*})$. Consequently the function $k$ is increasing on $[a_{*}, y_{1})$ and  then decreasing  on $(y_{1}, b_{*})$. On the other hand, we have $k(b_{*}) =0$, and thanks to \eqref{e2-1st-order},  $k(a_{*}) = \frac{c_{2}-c_{1}}{2 s(a_{*})} > 0$. Therefore we  again  have  the claim   that $k(x) \ge 0$ for all  $x\in [a_{*}, b_{*}] $.

Finally, to show that $v'(x) \le c_{2}$ on $[a_{*}, b_{*}]$,  we consider the function \begin{displaymath}
\rho(x): = k(x) + \frac{c_{1}-c_{2}}{2s(x)}, \quad x\in [a_{*}, b_{*}].
\end{displaymath} Note that  $\rho(b_{*}) = \frac{c_{1}-c_{2}}{2s(b_{*})} < 0 $. In addition,  $\rho(a_{*}) = \ell(a_{*}) =0 $ thanks to the proof of Proposition \ref{prop-a*b*}. Next we  compute \begin{displaymath}
\rho'(x) = -(\pi_{1}  (x)  - \lambda_{*} ) m(x) +  \frac{c_{1}-c_{2}}{2}\cdot \frac{2\mu (x) }{\sigma^{2}(x) s(x)} = (\lambda_{*}-\pi_{2}(x)) m(x).
\end{displaymath} As a result, if $\lambda_{*}-\pi_{2}(x) < 0$ for all $x\in [a_{*}, b_{*}] $, then $\rho(x)  \le \rho(a_{*}) = 0$   for all $x\in [a_{*}, b_{*}] $. Otherwise, using the monotonicity assumption of $\pi_{2}$ in Assumption \ref{assumpt-pi12-sect7}, there exists some $y_{2} \in (a_{*}, b_{*})$ so that $\lambda_{*}-\pi_{2}(x)  <0 $ for $x\in [a_{*}, y_{2})$ and $> 0$ for $x\in (y_{2}, b_{*}]$. This, in turn, implies that $\rho(x)$ is decreasing on $ [a_{*}, y_{2})$ and increasing on $ (y_{2}, b_{*}]$. In such a case, we still have $\rho(x)  \le  0$   for all $x\in [a_{*}, b_{*}] $. Then it follows that $ 2 s(x) k(x)  \le c_{2}- c_{1}$ and hence  $v'(x) =  c_{1} + 2 s(x) k(x) \le c_{2} $.

To summarize, we have shown that the function $u$ of \eqref{e:soln-HJB} is $C^{2}$ and satisfies
\begin{displaymath}
\begin{cases}
   \LL u(x) + h(x) - \lambda_{*} =0, \ \ c_{1} \le u'(x) \le c_{2},   & x\in (a_{*}, b_{*}), \\
      \LL u(x) + h(x) - \lambda_{*} \le 0, \ \   u'(x) = c_{2}, &  x\le a_{*}, \\
         \LL u(x) + h(x) - \lambda_{*} \le 0, \ \   u'(x) = c_{1}, &  x\ge b_{*}.
\end{cases}
\end{displaymath} In particular,   $(u, \lambda_{*})$ satisfies the HJB equation \ref{e:HJB-lta}. This concludes the  proof.
 \end{proof}

Now we present the main result of this section.

\begin{theorem}\label{thm-sect7} Let Assumption \ref{assumpt-pi12-sect7} hold. Then there exist $0< a_{*} < b_{*} < \infty$ so that the $(a_{*}, b_{*})$-reflection policy is optimal for problem \ref{def: lambda_0}. In addition,  $\lambda_{0} = \lambda_{*}$, where $\lambda_{*}$ is defined in \eqref{eq:lambda*defn}.
\end{theorem}
\begin{proof} Since $(u, \lambda_{*})$ is a solution to the HJB equation, where the function $u$ of \eqref{e:soln-HJB}  is bounded from below, it follows from Theorem \ref{thm-verification-lta} that $\lambda_{0} \le \lambda_{*}$. Furthermore,  the proof of Corollary \ref{cor-lambda_0>0} says that the  long-term average reward of the  $(a_{*}, b_{*})$-policy is equal to  $\lambda_{*}$ and therefore optimal. The  proof is complete.
\end{proof}	

\begin{example}\label{exam-VP-diffusion}
In this example, we consider the ergodic two-sided singular control problem for the  Verhulst-Pearl diffusion:
\begin{displaymath}
\d X(t) = \mu X(t) (1 - \gamma X(t))  \d t+  \sigma X(t) \d W( t)+ \d \xi(t) -\d\eta(t),
\end{displaymath} where $\mu > 0$ is the per-capita growth rate, $\frac{1}{\gamma} > 0$ is the carrying capacity, and $\sigma^{2}>0$ is  the variance parameter measuring the fluctuations in the per-capita growth rate. The singular control $\vphi: = \xi -\eta$ is as in Section \ref{sect-formulation}. Here, we can regard $\xi(t)$ and $\eta(t)$ as the cumulative renewing and harvesting amount up to time $t$, respectively. The scale and speed densities are
\begin{displaymath}
s(x) = x^{-\alpha} e^{\gamma \alpha (x-1)}, \quad m(x) = \frac{1}{\sigma^{2}} x^{\alpha -2} e^{-\gamma \alpha (x-1)}, \quad x >0.
\end{displaymath} where $\alpha=\frac{2\mu}{\sigma^{2}}$. Detailed computations using the 
criteria  given in Chapter 15 of \cite{KarlinT81} reveal that both 0 and $\infty$ are natural boundary points.

For positive constants $c_{1} < c_{2}$ and  a nonnegative function $h$ satisfying Assumption \ref{h-assumption} (i) and (ii), we consider the problem
\begin{equation}
\label{e:lambda_0-Verhulst-Pearl}
\lambda_{0}: = \sup_{\vphi(\cdot)\in \A_{x}} \liminf_{T\to\infty} \frac1T \E_{x}\bigg[ \int_{0}^{T} h(X(t))\d t + c_{1}\eta( T) - c_{2} \xi(T)\bigg],
\end{equation} where $\A_{x}$ is the set of admissible controls    as defined in \eqref{e:set A defn}.

We now claim that Assumption \ref{assumpt-pi12-sect7} is satisfied and hence in view of Theorem \ref{thm-sect7},  the optimal value $\lambda_{0}$ of \eqref{e:lambda_0-Verhulst-Pearl} is achieved by the $(a_{*}, b_{*})$-reflection policy, where $0< a_{*}< b_{*} < \infty$. Assumption \ref{assumpt-pi12-sect7} (i)  is obviously  satisfied. We now verify Assumption \ref{assumpt-pi12-sect7} (ii). For $i=1,2$, we have  $\pi_{i} (x) = h(x) +c_{i} \mu x(1-\gamma x)$ and hence $$\pi'_{i}(x) = h'(x) +c_{i}\mu -2c_{i}\mu\gamma x,\ \quad  x\ge 0.$$ Noting that $h'(x) > 0$ and $c_{i}, \mu > 0$, we see from the above equation that  $\pi'_{i}(x) > 0$ for $x> 0$ in a neighborhood of 0. Since $h$ satisfies  Assumption \ref{h-assumption} (i) and (ii), it follows that there exists some positive number $M$ so that $h'(x) < 1$ for all $x\ge M$. Thus $\lim_{x\to\infty}\pi'_{i} (x) = -\infty$. Then the continuity of $\pi'_{i}(x)$ implies that there exists a $\hx_{i}>0$ so that $\pi'_{i}(\hx_{i})=0$. Furthermore, since $h$ is strictly concave, $h'(x)$ is deceasing. Therefore  $\pi'_{i}(x)< 0$ for all $x> \hx_{i}$. Hence $\pi_{i}$ is first increasing on $[0, \hx_{i})$ and then decreasing on $[\hx_{i},\infty)$. On the other hand, using  Assumption \ref{h-assumption}  (ii), we have $\lim_{x\to\infty} \pi_{i}(x)=-\infty$ and hence there exists a $b_{0} > \hx_{1}$ so that $\pi_{1}(b_{0}) =0 $.  Assumption \ref{assumpt-pi12-sect7} (ii) is verified. 

Finally, note that the arguments in the previous paragraph also reveal that   $\pi_{1}$ is strictly positive on $(0, b_{0})$. This, together with facts that  $\pi_{1}(b_{0}) =0$ and $\lim_{a\downarrow 0} \frac1{s(a)} = 0$,  leads to  \eqref{e:key-condition} and hence Assumption \ref{assumpt-pi12-sect7} (iii) is in force. The proof is complete.
\end{example}



\end{document}